\documentclass[10pt]{article}

\usepackage{amsmath}\multlinegap=0pt
\usepackage[latin1]{inputenc}
\usepackage{amsthm}
\usepackage{amssymb}
\usepackage[francais,english]{babel}

\usepackage{amsmath}\multlinegap=0pt
\usepackage[latin1]{inputenc}
\usepackage{amsthm}
\usepackage{amssymb}
\usepackage[francais,english]{babel}
\usepackage{mathtools}

\numberwithin{equation}{section}
\theoremstyle{plain}
\newtheorem{thm}{Theorem}[section]

\newtheorem{prop}[thm]{Proposition}
\newtheorem{lem}[thm]{Lemma}

\newtheorem{dfn}[thm]{Definition}

\theoremstyle{definition}

\newtheorem{ex}[thm]{Example}
\theoremstyle{remark}

\usepackage{mathrsfs}
\usepackage{enumitem}

\title{A Self-dual Variational Approach to Stochastic Partial Differential Equations }

\author {Shirin Boroushaki\thanks{\scriptsize  \texttt{shirinbr@math.ubc.ca}.
This work is part of a PhD thesis prepared by this author under the supervision of N. Ghoussoub. Department of Mathematics, 1984 Mathematics Road,
  University of British Columbia, BC, V6T 1Z2, CANADA. 
 }\; 
 and Nassif Ghoussoub\thanks{\scriptsize \texttt{nassif@math.ubc.ca}. Both authors were partially supportted by a grant from the Natural Sciences and Engineering Research Council of Canada.
}}

\begin{document}

\def\pR{\R\cup\{+\infty\}}

\newcommand{\ntx}{\textnormal}
\newcommand\N{{\mathbb N}}
\newcommand\Z{{\mathbb Z}}
\newcommand\Q{{\mathbb Q}}
\newcommand\R{{\mathbb R}}
\newcommand\E{{\mathbb E}}
\newcommand\mL{{\mathcal{L}}}
\newcommand{\mB}{\mathcal{B}}

\def\F{{\mathcal {F}}}
\def\e{{\mathcal {E}}}
\def\P{{\mathbb P}}
\def\L{{\mathcal {L}}}
\def\X{{\mathcal {X}}}
\def\div{\textnormal{div} }
\def\inn{\textnormal{in} }
\def\on{\textnormal{on} }
\def\msX{\mathscr{X}}

\def\msE{\mathscr{E}}
\def\msM{\mathscr{M}}
\def\msH{\mathscr{H}}
\def\msF{\mathscr{F}}
\def\msL{\mathcal{L}}
\def \mbL{\mathbb{L}}
\def\vr{\vert}
\def\Vr{\Vert}

\def\F{{\cal F}}
\def\G{{\cal G}}
\def\L{{\cal L}}
\def\Rm{{\cal R}}
\def\div{\mbox{div} }
\def\t{\tilde}
\def\t{\tilde}

\maketitle

\begin{abstract} Unlike many of their deterministic counterparts, stochastic partial differential equations are not amenable to the methods of calculus of variations \`a la Euler-Lagrange. In this paper, we show how self-dual variational calculus leads to variational solutions of various stochastic partial differential equations driven by monotone vector fields. We construct solutions as minima of suitable non-negative and self-dual energy functionals on It\^o spaces of stochastic processes. We show how a stochastic version of Bolza's duality leads to solutions for equations with additive noise. We then use a Hamiltonian formulation to construct solutions for non-linear equations with non-additive noise such as the stochastic Navier-Stokes equations in dimension two. 
\end{abstract}


\section{Introduction}
Self-dual variational calculus was developed in the last fifteen years in an effort to construct variational solutions to various  partial differential equations and evolutions, that do not fall in the Euler-Lagrange framework of the standard calculus of variations. We refer to the monograph \cite{G-book} for a comprehensive account of that theory. In this paper, we show how such a calculus can be applied to solve stochastic partial differential equations, which clearly do not fit in Euler-Lagrange theory, since their solutions are not known to be critical points of energy functionals. We show here that at least for some of these equations, solutions can be obtained as minima of suitable self-dual functionals on It\^o spaces of random processes.  

The self-dual variational approach applies whenever stochastic partial differential equations are driven by monotone vector fields. These are operators $A: D(A)\subset V\to V^*$ --possibly set-valued-- from a  Banach space $V$ into its dual, that satisfy
\begin{equation}
\langle p-q, u-v\rangle \geq 0\quad \hbox{for all $(u,p)$ and $(v,q)$ on the graph of $A$.}
\end{equation} 
We shall be able to tackle SPDEs of the following form 
  \begin{equation}\label{three-SDE}
\begin{cases}
du(t)= -A (t,u(t))\, dt - \Lambda u(t)\, dt+B(t, u(t)) dW(t)\\
\, \,  u(0)=u_0, 
\end{cases}
\end{equation}
where $u_0 \in L^2(\Omega, \F_0,\P; H)$, $H$ being a Hilbert space, and $W(t)$ is a real-valued Wiener process on a complete probability space $(\Omega,\F,\P)$ with normal filtration $(\F_t)_t$. The vector field $A: \Omega \times [0, T] \times V \rightarrow V^*$ can be a time-dependent adapted random --possibly set-valued-- maximal monotone map, where $V$ is a Banach space such that $V\subset H\subset V^*$ constitute a Gelfand triple. The operators $\Lambda: V \rightarrow V^*$ and $B: [0, T] \times V \rightarrow H$ will be suitable linear or non-linear (progressively measurable in the case of $B$) maps. 
By solutions, we shall mean stochastic processes $u$
that verify the integral equation
 \begin{equation*}
u(t)=u_0-\int_0^t A(s,u(s))ds -\int_0^t \Lambda u(s)ds+ \int_0^t B(s,u(s))\, dW(s),
\end{equation*} 
where the stochastic integral is in the sense of It\^o. The cases when $B(t, \cdot)$ is a prescribed $H$-valued progressively measurable random process $B(t)$ are referred to as SPDEs with additive noise. \\
Self-dual variational calculus will allow us to deal with SPDEs in divergence form such as:
\begin{equation}\label{two-SDE}
\begin{cases}
{du (t)}= \div(\beta (\nabla u(t,x)))dt +B(t, u(t))dW(t)  \qquad & \textnormal{in} \ [0,T] \times D\\
u(t, x)=0 \qquad & \textnormal{in} \ [0,T] \times \partial D\\
u(0,x)=u_0(x) &\on \ D, 
\end{cases}
\end{equation}
where here, $y\to \beta (y)$ is a progressively measurable monotone vector field on $\R^n$, and 
$D$ is a bounded domain in $\R^n$. 

The genesis of self-dual variational calculus can be traced to a 1970 paper of Brezis-Ekeland \cite{BE1, BE2} (see also Nayroles \cite{Nayroles76, Nayroles76b}), where they proposed a variational principle for the heat equation and other gradient flows for 
convex energies. The conjecture was eventually verified by Ghoussoub-Tzou \cite{GT1}, who identified and exploited the self-dual nature of the Lagrangians involved. Since then, the theory was developed in many directions \cite{G1, G2, GM}, so as to provide existence results for several stationary and parabolic -but so far deterministic- PDEs, which may or may not be of Euler-Lagrange type.
 
While in most examples where the approach was used, the self-dual Lagrangians were explicit, 
an important development in the theory was the realization \cite{G4} that in a prior work, Fitzpatrick \cite{F} had associated a (somewhat) self-dual Lagrangian to any given monotone vector field. That meant that the variational theory could apply to any equation involving such operators. We refer to the monograph \cite{G-book} for a survey and for applications to existence results for solutions of several PDEs and evolution equations. 
We also note that since the appearance of this monograph, the theory has been successfully applied to the homogenization of periodic non-self adjoint problems (Ghoussoub-Moameni-Zarate \cite{GMS}). More recently, the self-dual approach was used in \cite{A1, A2}
 to tackle the more general problem of stochastic homogenization of such equations and to provide valuable quantitative estimates. 

The application of the method  to solving SPDEs is long overdue, though V. Barbu \cite{Bar} did use a Brezis-Ekeland approach to address SPDEs driven by gradients of a convex function and additive noise. After an earlier version of this paper was published, he notified us of a more recent work of his with R\"ockner \cite{BR} that deals with a non-additive but linear noise. We shall deal here with more general situations. We note that some of the equations below have already been solved by other methods, starting with the celebrated thesis of Pardoux \cite{pardoux2}, and many other subsequent works \cite{DPZ, par1, Pardoux1, PR}. This paper is about presenting a new variational approach, hoping it will lead to progress on other unresolved equations.  
 
To introduce the method, we consider the simplest example, where the monotone operator $A$ is given by the gradient $\partial \varphi$ of a (possibly random and progressively measurable) function $\varphi:  [0,T] \times H \rightarrow \pR$ such that for every $t\in [0,T]$,  
 the function $\varphi(t, \cdot)$ is convex and lower semi-continuous on a Hilbert space $H$, and the stochastics is driven by a given progressively measurable additive noise coefficient $B:\Omega\times [0,T] \rightarrow H$. The equation becomes
\begin{equation}\label{basic}
\begin{cases}
du(t)= -\partial \varphi (t,u(t)) dt + B(t) dW(t)\\
u(0)=u_0.%
\end{cases}
\end{equation}
We consider the  following {\it It\^o space over $H$}, 
\begin{equation*}
 \mathcal{A}^2_H= \left\{u: \Omega_T \rightarrow H ; \  u(t) = u(0) + \int_0^t \tilde{u}(s) ds+ \int_0^t F_u(s) dW(s)\right\},
\end{equation*} 
where
$u(0) \in L^2(\Omega,\F_0,\mathbb{P}; H)$, $ \tilde{u} \in L^2(\Omega_T; H)$ and $F_u \in L^2(\Omega_T; H)$, where $\Omega_T=\Omega \times [0, T].$
Here, both the drift $\tilde{u}$ and the diffusive term $F_u$ are progressively measurable. 
 The key idea is that a solution for (\ref{basic}) can be obtained by minimizing 
 the following functional on $\mathcal{A}^2_H$, 
\begin{equation*}
{I}(u)=\E \, \Big\{\int_0^T L_\varphi(u(t), -\tilde{u}(t)) \, dt+ \frac{1}{2} \int_0^T M_B(F_u(t),-F_u(t)) \, dt + \ell_{u_0}(u(0),u(T))\Big\},
\end{equation*} 
where 
\begin{itemize}
\item $L_\varphi$ is the (possibly random) time-dependent Lagrangian on $H\times H$ given by 
$$L_\varphi (u, p)= \varphi(w,t,u)+ \varphi^*(w, t,p),$$ 
where $\varphi^*$ is the Legendre transform of the convex function $\varphi$;

\item  $\ell_{u_0}$ is the time-boundary random Lagrangian on $H\times H$ given by 
$$\ell_{u_0}(a,b):=\ell_{u_0(w)}(a,b)=\frac{1}{2}  \Vr a\Vr^{^2}_H+ \frac{1}{2}  \Vr b\Vr^{^2}_H-2\langle u_0(w), a\rangle_H+\Vr u_0(w)\Vr^{^2}_H;$$
\item $M_B$ is the random time-dependent diffusive Lagrangian on $H\times H$, given by
$$M_B(G_1,G_2):=\Psi_{B(w, t)}(G_1)+\Psi_{B(w,t)}^*(G_2),$$ 
where $\Psi_{B(w, t)}:H \rightarrow \pR$ is the convex function  
$\Psi_{B(w, t)} (G)=\frac{1}{2} \Vr G-2B(w, t)\Vr^2_H.$

\end{itemize}
We note that it is not sufficient that $I$ attains its infimum on $\mathcal{A}^2_H$ at some $v$, but one needs to also show that the infimum is actually equal to zero, so as to obtain 
\begin{align*}
0=I(v)& =\  \E \int_0^T  \Big( \varphi(t,v)+ \varphi^*(t,-\tilde{v}(t))\Big) \, dt \\
& \quad+ \E\ \Big(\frac{1}{2}  \Vr v(0)\Vr^2_H+ \frac{1}{2}  \Vr v(T)\Vr^2_H -2\langle u_0, v(0)\rangle +\Vr u_0\Vr^2_H\Big) \\
&\quad + \E \int_0^T\Big(\frac{1}{2} \, \Vr F_v(t)-2B(t)\Vr_H^2\,+ \,\frac{1}{2} \, \Vr F_v(t) \Vr_H^2 -2 \langle F_v(t),B(t) \rangle \Big) \, dt,
\end{align*}
where we have used the fact that $\Psi_B^*(G)=\frac{1}{2} \, \Vr G \Vr_H^2 +2 \langle G,B\rangle_H$. \\
By using It\^o's formula, and by adding and subtracting the term $\E \int_0^T \langle v(t), \tilde{v}(t) \rangle dt$, we can rewrite $I(v)$ as the sum of 3 non-negative terms
\begin{align*}
\begin{split}
0=I(v)&=\E \int_0^T  \Big( \varphi(t,v)+ \varphi^*(t,-\tilde{v}(t))+\langle v(t), \tilde{v}(t) \rangle\Big) \, dt\\ 
&\quad + 2 \, \E \int_0^T \Vr F_v-B\Vr^2_H \, dt+ \E \ \Vr v(0)-u_0\Vr^2_H,
\end{split}
\end{align*}
which  yields that for almost all $t\in [0,T]$, $\P$-a.s.
$$ \varphi(t,v)+ \varphi^*(t,-\tilde{v}(t))+\langle v(t), \tilde{v}(t) \rangle=0,\hbox{
hence 
$-\tilde{v}(t)\in \partial \varphi(v(t)).$ }
$$
The two other identities readily give that $B=F_v$ and $v(0)=u_0$. In other words,
$v\in \mathcal{A}_H^2$, and satisfies (\ref{basic}). 
The self-dual variational calculus allows to apply the above approach in much more generality since the  special Lagrangians $L_\varphi$, $\ell_{u_0}$ and $M$ can be replaced by much more general self-dual Lagrangians. \\
In Section \ref{SD}, we shall collect --for the convenience of the reader-- the elements of self-dual variational theory that will be needed in the proofs. In Section \ref{lift}, we show how one can lift self-dual Lagrangians from state space to function spaces and then to  It\^o spaces of stochastic processes. In Section \ref{additive}, we give a variational resolution for basic SPDEs involving additive noise, such as
 \begin{equation}\label{one-SDE}
\begin{cases}
du(t)= -A (t,u(t)) dt + B(t) dW(t)\\
u(0)=u_0,
\end{cases}
\end{equation}
by establishing a stochastic version of the well known Bolza duality, which we believe is interesting in its own right as it may have applications to stochastic control problems.\\ 
Section \ref{Ex} contains applications 
to classical SPDEs such as the following stochastic evolution driven by a diffusion and a transport operator,
  \begin{equation}\label{transport_stoch.0}
\begin{cases}
du = (\Delta u  + \textbf{a}(x)\cdot \nabla u) dt+ B(t) dW \ & \textnormal{on} \ [0,T] \times  D \\
u(0)=u_0 & \textnormal{on} \ D,
\end{cases}
\end{equation} 
where 
$\textbf{a}: D \rightarrow \R^n$ is a smooth vector field with compact support in $D$, such that $\div(\textbf{a}) \geq 0$.\\
In Section \ref{non-add}, we deal with quite general SPDEs driven by  a self-dual Lagrangian on $L^\alpha(\Omega_T;V)\times L^\beta(\Omega_T;V^*)$ and a non-additive noise.  We then apply this result in Section \ref{last.section} to resolve equations of the form (\ref{three-SDE}) and (\ref{two-SDE}), such as 
\begin{equation}\label{exm-delta}
\begin{cases}
du(t)=\Delta u \, dt+ |u|^{q-1}u \, dW(t) \qquad & \textnormal{in} \ [0,T] \times D\\
u(t,x)=0 &\textnormal{in} \ [0,T] \times \partial D\\
u(0,x)=u_0(x) &\on \  D,  
\end{cases}
\end{equation}
where $\frac{1}{2}\leq q < \frac{n}{n-2}$, and 
\begin{equation}\label{exm-div}
\begin{cases}
{du}= \div(\beta(\nabla u(t,x)))dt + B(t,u(t))\, dW(t)  \qquad & \textnormal{in} \ [0,T] \times D\\
u(t,x)=0 &\textnormal{in} \ [0,T] \times \partial D\\
u(0,x)=u_0(x) &\on \  D, 
\end{cases}
\end{equation}
where 
$D$ is a bounded domain in $\R^n$ and the initial position  $u_0$ belongs to $L^2(\Omega, \F_0,\P; L^2(D))$.
We shall also deal with the stochastic Navier-Stokes equations in two dimensions, 
\begin{equation}
\begin{cases}
du(t)=\Delta u \, dt+ u\cdot \nabla u + \nabla p+ |u|^{q-1}u \, dW(t) \qquad & \textnormal{in} \ [0,T] \times D\\
\div(u)=0 & \textnormal{on} \ D\\
u(t,x)=0 &\textnormal{in} \ [0,T] \times \partial D\\
u(0,x)=u_0(x) &\on \ D,
\end{cases}
\end{equation}
where $\frac{1}{2}\leq q \leq 1$.

  
\section{Elements of self-dual variational calculus \label{SD}}

If $V$ is a reflexive Banach space and $V^*$ is its dual, then a  (jointly) convex lower semi-continuous Lagrangian $L: V \times V^* \rightarrow \R \cup \{+\infty\}$ is said to be \textit{self-dual} on $V \times V^*$ if 
\begin{equation}
L^*(p,u)=L(u,p), \qquad (u,p) \in V \times V^*,
\end{equation}
where $L^*$ is the Fenchel-Legendre dual of $L$ in both variables, i.e.,
$$L^*(q,v)= \sup \{\langle q,u \rangle + \langle v, p \rangle - L(u,p); \ u \in V,
p \in V^*\}.$$
Such Lagrangians satisfy the following basic property $$L(u,p)-\langle u,p \rangle \geq 0 , \qquad \forall \ (u,p) \in V \times V^*.$$
We are interested in the case when the above is an equality, hence we consider the corresponding --possibly multivalued-- \textit{self-dual vector field} $\bar{\partial}L:V \rightarrow 2^{V^*}$ defined for each $u\in V$ as the --possibly empty-- subset $\bar{\partial}L(u)$ of $V^*$ given by
$$\bar{\partial}L(u)=\{p\in V^*;\ L(u,p)- \langle u,p \rangle =0 \}=\{p\in V^*;\ (p,u) \in \partial L(u,p) \},
$$ where $\partial L$ is the subdifferential of the convex function $L$.

\subsection{Self-dual Lagrangians as potentials for monotone vector fields}

Self-dual vector fields are natural extensions of subdifferentials of convex lower semi-continuous functions. Indeed, the most basic self-dual Lagrangians are of the form $L(u,p) = \varphi(u) + \varphi^*(p)$ where $\varphi$ is a convex function on $V$, and $\varphi^*$ is its Fenchel dual on $V^*$ (i.e.,
$\varphi^*(p)= \sup\{\langle u,p\rangle -\varphi(u), u\in V \}$) for which
$$\bar{\partial}L(u)=\partial \varphi(u).$$
Other examples of self-dual Lagrangians are of the form $L(u,p) = \varphi(u) + \varphi^*(-\Gamma u+p)$ where $\Gamma: V \rightarrow V^*$ is a skew-adjoint operator. The
corresponding self-dual vector field is then $$\bar{\partial}L(u)=\partial \varphi(u)+\Gamma u.$$
Actually, both $\partial \varphi $ and $\partial \varphi+\Gamma $ are particular examples of {\it maximal monotone} operators, which
are set-valued maps $A:V \rightarrow 2^{V^*}$ whose graphs in $V \times V^*$ are maximal (for set inclusion) among all monotone
subsets $G$ of $V \times V^*$. In fact, it turned out that maximal monotone operators and self-dual vector fields are essentially the same. The following was first noted by Fitzpatrick \cite{F} (with a weaker notion of (sub) self-duality), and re-discovered and strengthened  later by various authors. See \cite{G-book} for details.

\begin{thm}\label{Fitz}
If $A: D(A) \subset V \rightarrow 2^{V^*}$ is a maximal monotone operator with a
non-empty domain, then there exists a self-dual Lagrangian $L$ on $V \times V^*$ such that $A = \bar{\partial}L$. Conversely, if $L$ is a proper self-dual Lagrangian on a reflexive Banach space $V \times V^*$, then the vector field $u\mapsto \bar{\partial}L(u)$ is maximal monotone.
\end{thm}

\noindent Another needed property of the class of self-dual Lagrangians is its  stability under convolution. 
\begin{lem} {\rm (\cite{G-book}  Proposition 3.4)}\label{conv}
If $L$ and $N$ are two self-dual Lagrangians on a reflexive Banach space $X \times X^*$ such that $\text{Dom}_1(L)-\text{Dom}_1(N)$ contains a neighborhood of the origin, then the Lagrangian defined by
$$(L\oplus N)(u,p)=\inf_{r\in X^*}\{L(u,r)+N(u,p-r)\}$$ is also self-dual on $X\times X^*$.
\end{lem}
As in deterministic evolution equations, one often aim for more regular solutions that are valued in suitable Sobolev spaces, as opposed to just $L^2$. Moreover,  the required coercivity condition (on the underlying Hilbert space) is quite restrictive and is not satisfied by most Lagrangians of interest.  
A natural setting is the so-called evolution triple of Gelfand, which consists of having a Hilbert space sandwiched between a reflexive Banach space $V$ and its dual $V^*$, i.e., 
$$V \subset H\cong H^* \subset V^*,$$
 where the injections are continuous and with dense range, in such a way that if $v\in V$ and $h\in H$, then $\langle v,h \rangle_H=\langle v,h \rangle_{_{V, V^*}}$. A typical evolution  triple is $V:=H^1_0(D)\subset H:=L^2(D) \subset V^*:= H^{-1}(D),$ where $D$ is a bounded domain in $\R^n$. The following lemma explains the connection between the self-duality on $H$ and $V$. 
\begin{lem}\label{lift-to-H}{\rm (\cite{G-book} Lemma 3.4)}
Let $V\subset H \subset V^*$ be an evolution triple, and suppose $L:V\times V^*\rightarrow \R\cup\{+\infty\}$ is a self-dual Lagrangian on the Banach space $V$, that satisfies for some $C_1, C_2>0$ and $r_1\geq r_2>1$,
$$C_2(\Vr u\Vr_V^{r_2}-1)\leq L(u,0) \leq C_1(1+\Vr u\Vr_V^{r_1}) \qquad \ntx{for all} \ u\in V.$$
Then, the Lagrangian defined on $H\times H$ by
$$\bar{L}(u,p):=
\begin{cases}
L(u,p) \quad &u\in V\\
+\infty &u\in H \backslash V
\end{cases}$$
is self-dual on the Hilbert space $H\times H$.
\end{lem}
 
\subsection{Two self-dual variational principles}

The basic premise of self-dual variational calculus is that several differential systems 
can be written in the form 
$0 \in \bar{\partial}L(u),$
 where $L$ is a self-dual Lagrangian on phase space $V \times V^*$. These are the \textit{completely self-dual systems}. A solution to these systems can be obtained as a minimizer of a \textit{completely self-dual functional} $I(u)=L(u, 0)$ for which the minimum value is 0. The following is the basic minimization principle for self-dual energy functionals.

\begin{thm}\label{var_prin} {\rm (\cite{G0})} Suppose ${X}$ is a reflexive Banach space, and 
let ${L}$ be a self-dual Lagrangian on ${X} \times {X}^* $ 
such that the mapping $u\rightarrow {L}(u,0)$ is coercive in the sense that
$\underset{\Vr u \Vr \rightarrow \infty} \lim \frac{{L}(u,0)}{\Vr u \Vr}= +\infty$.
Then, there exists $\bar{u} \in {X}$ such that $I(\bar{u})= \inf\limits_{u \in {X}}{L}(u,0)=0.$
\end{thm}
\noindent As noted in \cite{G0}, it actually suffices that ${L}$ be {\it partially self-dual}, that is if
$${L}^*(0,u)={L}(u,0) \quad \hbox{for every $u\in {X}$}.$$
 
We shall also need the Hamiltonian associated to a self-dual Lagrangian, that is the functional on $X\times X$ defined as
$H_L: X \times X \rightarrow \R\cup\{-\infty\} \cup \{+\infty\}$
$$H_L(u,v)=\sup_{p \in V^*} \{\langle v,p \rangle - L(u,p)\}, 
$$
which is the Legendre transform in the second 
variable. It is  easy to see that if $L$ is a self-dual Lagrangian on $X\times X^*$, then its Hamiltonian on $X\times X$ 
satisfies the following properties:
\begin{itemize}
\item $H_L$ 
 is concave in $u$ and convex lower semi-continuous in $v$.
 
\item 
$H_L(v,u) \leq -H_L(u,v)$ for all $u,v\in X$. 
\end{itemize}
As established in \cite{G3}, the Hamiltonian formulation allows for the minimization of direct sums of self-dual functionals.
The following variational principle is useful in the case when non-linear and unbounded operators are involved. 
 
\begin{thm}\label{superpose-thm} {\rm  (\cite{G-book})} Consider three reflexive Banach spaces $Z,X_1,X_2$ and operators $A_1 : D(A_1) \subset Z \rightarrow X_1$ , $\Gamma_1 : D(\Gamma_1) \subset Z \rightarrow X_1^*$, $A_2 : D(A_2) \subset Z \rightarrow X_2$, and  $\Gamma_2 : D(\Gamma_2) \subset Z \rightarrow X_2^*$,
such that $A_1$ and $A_2$ are linear, while $\Gamma_1$ and $\Gamma_2$ --not necessarily linear-- are weak-to-weak continuous.
Suppose $G$ is a closed linear subspace of $Z$ such that $G \subset D(A_1) \cap D(A_2)\cap D(\Gamma_1)\cap D(\Gamma_2)$, while the following properties are satisfied:
\begin{enumerate}
\item The image of $G_0:=\ntx{Ker}(A_2) \cap G$ by $A_1$ is dense in $X_1$.
\item The image of $G$ by $A_2$ is dense in $X_2$.
\item $ u \mapsto \langle A_1 u, \Gamma_1 u \rangle+\langle A_2 u, \Gamma_2 u \rangle$ is weakly upper semi-continuous on $G$.
\end{enumerate}
Let $L_i, i=1,2$ be self-dual Lagrangians on $X_i\times X_i^*$ such that the Hamiltonians ${H}_{L_i}$ are continuous in the first variable on $X_i$. Under the following coercivity condition,
\begin{equation}\label{superpose-coerc}
\lim_{\substack{\Vr u\Vr \rightarrow \infty \\ u\in G}} {H}_{L_1}(0,A_1 u) - \langle A_1 u,\Gamma_1 u \rangle +{H}_{L_2}(0,A_2 u) - \langle A_2 u,\Gamma_2 u \rangle=+\infty,
\end{equation}
the functional
$$I(u)= L_1(A_1 u,\Gamma_1 u) - \langle A_1 u,\Gamma_1 u \rangle+L_2(A_2 u,\Gamma_2 u) - \langle A_2 u,\Gamma_2 u \rangle$$
attains its minimum at a point $v\in G$ such that $I(v)=0$, and 
\begin{equation*}
\Gamma_1(v) \in \bar{\partial}L_1(A_1 v),
\end{equation*}
\begin{equation}
\Gamma_2(v) \in \bar{\partial}L_2(A_2 v).
\end{equation}
\end{thm}

\section{Lifting  random self-dual Lagrangians to It\^o path spaces\label{lift}}
Let $V$ be a reflexive Banach space, and $T \in [0,\infty)$ be fixed. Consider a complete probability space $(\Omega, \F, \P)$ with a normal filtration $\F_t ,\  t \in [0, T]$, and let  $L^\alpha(\Omega \times [0,T]; V)$ be the space of Bochner integrable functions from $\Omega_T:= \Omega \times [0,T]$ into $V$ with the norm 
$\Vert u \Vert_{L^\alpha_V}^\alpha := \E \int_0^T \Vert u(t) \Vert_V^\alpha \ dt.$ 
We may use the shorter notation $L^\alpha_V(\Omega_T):= L^\alpha(\Omega \times [0,T]; V)$ in the sequel.
\begin{dfn}
A self-dual $\Omega_T$-dependent convex Lagrangian on $V \times V^*$ is a function $L:\Omega_T \times V \times V^* \rightarrow \pR$ such that:
\begin{enumerate}
\item $L$ is progressively measurable with respect to the $\sigma$-field generated by the products of $\F_t$ and Borel sets in $[0,t]$ and $V\times V^*$, i.e. for every $t\in [0,T]$, $L(t, \cdot, \cdot) $ is $\F_t \otimes \mathcal{B}([0,t]) \otimes \mathcal{B}(V)\otimes \mathcal{B}(V^*)$-measurable.
\item For each $t \in [0,T]$, $\P$-a.s. the function $L(t, \cdot,\cdot)$ is convex and lower semi-continuous  on $V \times V^*$.
\item For any $t\in [0,T]$, we have $\P$-a.s.
$L^*(t,p,u)=L(t,u,p)$ for all $(u,p)\in V \times V^*,$
where $L^*$ is the Legendre transform of $L$ in the last two variables.
\end{enumerate}
\end{dfn}
\noindent To each $\Omega_T$-dependent Lagrangian $L$ on $\Omega_T \times V \times V^*$, one can associate the corresponding Lagrangian $\mathcal{L}$ on the path space $L^\alpha_V(\Omega_T) \times L^\beta_{V^*}(\Omega_T)$, where $\frac{1}{\alpha}+ \frac{1}{\beta}=1$, to be
$$\mathcal{L}(u,p):= \E \int_0^T L(t, u(t), p(t)) \ dt,$$
with the duality between $L^\alpha_V(\Omega_T)$ and  $L^\beta_{V^*}(\Omega_T)$ given by 
$\langle u,p \rangle= \E \int_0^T \langle u(t),p(t) \rangle_{_{V,V^*}} dt.$
The associated Hamiltonian on $L^\alpha_V(\Omega_T) \times L^\alpha_V(\Omega_T)$ will then be 
$$H_{\mathcal{L}}(u,v) = \sup \Big\{\E \int_0^T \{ \langle v(t),p(t) \rangle - L(t, u(t), p(t))\}dt \ ; \ p \in L^\beta_{V^*}(\Omega_T) \Big\}.$$
The Legendre dual of a "lifted" Lagrangian in both variables naturally lifts to the space of paths $L^\alpha_V(\Omega_T) \times L^\beta_{V^*}(\Omega_T)$ via
\begin{equation*}\label{path_leg_1}
\mathcal{L}^*(q,v) = \sup_{\substack{u \in L^\alpha_V(\Omega_T) \\ p \in  L^\beta_{V^*}(\Omega_T) }} \Big\{ \E \int_0^T \{\langle q(t),u(t) \rangle + \langle v(t),p(t) \rangle - L( t, u(t), p(t))\} \, dt\Big\}.
\end{equation*}
The following proposition is standard. Under suitable boundedness conditions (see for example \cite{Ek}), if 
$L$ is an $\Omega_T$-dependent Lagrangian on $V \times V^*$, and $\mathcal{L}$ is the corresponding Lagrangian on the process space $L^\alpha_V(\Omega_T)  \times L^\beta_{V^*}(\Omega_T) $, then,
\begin{equation}
\hbox{$\mathcal{L}^*(p,u)= \E \int_0^T L^*(t, p(t), u(t)) dt$\quad and \quad 
  $H_{\mathcal{L}}(u,v)= \E \int_0^T H_L(t, u(t), v(t)) dt$.
  }
  \end{equation}
 

\subsection{Self-dual Lagrangians associated to progressively measurable monotone fields \label{mon.rep}}

Consider now a progressively measurable --possibly set-valued-- maximal monotone map that is a map $A:\Omega_T \times V \rightarrow 2^{V^*}$ that is measurable for each $t$, with respect to the product $\sigma$-field $\F_t \otimes \mathcal{B}([0,t]) \otimes \mathcal{B}(V)$, and such that for each $t\in [0,T]$, $\P$-a.s., the vector field $A_{\omega,t}:=A(t, \omega, \cdot,\cdot)$ is maximal monotone on $V$. By Theorem \ref{Fitz}, one can associate 
to the maximal monotone maps $A_{\omega,t}$, self-dual Lagrangians 
$L_{A_{\omega,t}}$ on $V \times V^*$, in such a way that \begin{equation*}
A_{\omega,t}=\bar{\partial}L_{A_{\omega,t}}\quad \hbox{for every $t\in[0,T]$, and $\P$-a.s.}
\end{equation*}
This correspondence can be done measurably in such a way that if $A$ is progressively measurable, then the same holds for the corresponding $\Omega_T$-dependent Lagrangian $L$. We can then lift the random Lagrangian to the space $L^\alpha_V(\Omega_T) \times L^\beta_{V^*}(\Omega_T)$ via 
\begin{equation*}
\mathcal{L}_A(u,p)=\E \int_0^T L_{A_{\omega,t}}(u(\omega,t),p(\omega,t))dt.
\end{equation*}
Boundedness and coercivity conditions on $A$ translate into corresponding conditions on the representing Lagrangians  as follows. For simplicity, we shall assume throughout that the monotone operators are single-valued, though the results apply for general vector fields. The following was proved in \cite{GMS}.
\begin{lem}\label{lem-AtoL} 
Let $A_{\omega,t}$ be the maximal monotone operator as above with the corresponding potential Lagrangian  $L_{A_{\omega,t}}$. Assume that for all $u\in V, dt \otimes \P \, a.s.$, $A_{\omega,t}$  satisfies
\begin{align}\label{A-cond2}
 \langle A_{\omega,t}u,u\rangle \geq \max \left\{  c_1(\omega,t)\Vr u \Vr_V^\alpha -m_1(\omega,t), \, c_2(\omega,t)\Vr A_{\omega,t}u \Vr^\beta_{V^*}-m_2(\omega,t)\right\},
\end{align}
\noindent where $c_1, c_2 \in L^\infty(\Omega_T, dt \otimes \P)$ and $m_1, m_2 \in L^1(\Omega_T, dt \otimes \P)$.
Then the corresponding Lagrangians satisfy the following:
\begin{equation*}
C_1(\omega,t) (\Vr u \Vr^\alpha_V+\Vr p \Vr^\beta_{V^*}-n_1(\omega,t)) \leq L_{A_{w,t}} (u,p)
\leq C_2(\omega,t)(\Vr u \Vr^\alpha_V+\Vr p \Vr^\beta_{V^*} +n_2(\omega,t)),
\end{equation*}
for some $C_1, C_2 \in L^\infty(\Omega_T)$ and $n_1, n_2\in L^1(\Omega_T)$.\\
The lifted Lagrangian on the $L^\alpha$-spaces then satisfy for some $C_1,C_2>0$, 
\begin{equation*}
 C_1 (\Vr u \Vr^\alpha_{L^\alpha_V(\Omega_T)}+\Vr p \Vr^\beta_{L^\beta_{V^*}(\Omega_T)}-1) \leq  \mathcal{L}_A(u,p)
\leq C_2(1+\Vr u \Vr^\alpha_{{\cal L}^\alpha_V(\Omega_T)}+\Vr p \Vr^\beta_{L^\beta_{V^*}(\Omega_T)}).
\end{equation*}

\end{lem}

\subsection{It\^o path spaces over a Hilbert space}

Suppose now that $U$ is a Hilbert space. 
For $t\in [0,T]$, a cylindrical Wiener process $W(t)$ in $U$ can be represented by  
$$W(t)=\sum_{k\in \mathbb{N}} \beta_k(t) \, e_k, \qquad t\geq 0,$$
where $\{\beta_k\}$ is a sequence of mutually independent Brownian motions on the filtered probability space and $\{e_k\}$ is an orthonormal basis in $U$. For simplicity, we shall 
 assume in the sequel that  $W$ is a real-valued Wiener process i.e. $U=\R$.
We now recall It\^o's formula. 

\begin{prop}\label{Ito-thm} {\rm (\cite{Pardoux1}, \cite{PR})} Let $H$ be a Hilbert space with $\langle \, , \, \rangle_H$ as its scalar product. Fix $x_0 \in L^2(\Omega,\F_0,\P;H)$, and let $y \in L^2(\Omega_T ;H)$, $Z\in L^2(\Omega_T;H)$ be two progressively measurable processes. Define the $H$-valued process $u$ as
\begin{align}\label{ito-process}
u(t):= x_0+ \int_0^t y(s) ds + \int_0^t Z(s)dW(s). 
\end{align}
Then, the following hold:
\begin{enumerate}
\item $u$ is a continuous $H$-valued adapted process such that $\E\left(\sup_{t\in [0,T]} \Vr u(t)\Vr_H^2 \right) < \infty$.
\item {\rm{(It\^o's formula)}} For all $t\in[0,T]$,
$$\Vr u(t)\Vr_H^2=\Vr x_0\Vr_H^2+2 \int_0^t \langle y(s), u(s) \rangle_H ds + \int_0^t \Vr Z(s)\Vr^2_H ds+ 2 \int_0^t \langle u(s), Z(s)\rangle_H dW(s),$$
\end{enumerate}
and consequently
\begin{equation*}\label{E_ito_formula}
\E(\Vr u(t) \Vr^2_H)= \E(\Vr x_0 \Vr^2_H)+\E\int_0^t \Big(2\langle y(s), u(s) \rangle_H+ \Vr Z(s) \Vr^2_H \Big)ds. 
\end{equation*}
\end{prop}
\noindent More generally, the following \textit{integration by parts} formula holds.  For two processes $u$ and $v$ of the form:
$$u(t) = u(0) + \int_0^t \tilde{u}(s) ds+ \int_0^t F_u(s) dW(s), \quad v(t)=v(0)+\int_0^t \tilde{v}(s) ds + \int_0^t G_v(s)dW(s),$$
we have 
\begin{align}\label{int_by_parts}
\nonumber\E \int_0^T \langle u(t), \tilde{v}(t) \rangle dt=&-\E \int_0^T \langle v(t), \tilde{u}(t) \rangle dt -\E \int_0^T \langle F_u(t), G_v(t)\rangle dt\\
& + \E \langle u(T), v(T) \rangle_H -\E \langle u(0), v(0) \rangle_H.
\end{align}
Now we define the {\it It\^o space} $\mathcal{A}^2_H$ consisting of all $H$-valued processes of the following form:
\begin{align}\label{spaceA22}
\begin{split}
\mathcal{A}^2_H= \Big\{u: &\Omega_T \rightarrow H ; \  u(t) = u(0) + \int_0^t \tilde{u}(s) ds+ \int_0^t F_u(s) dW(s), \\
&\textnormal{for} \ u(0) \in  L^2(\Omega,\F_0,\P;H), \ \tilde{u} \in L^2(\Omega_T; H), \, F_u \in L^2(\Omega_T; H)\Big\},
\end{split}
\end{align}
where $\tilde{u}$ and $F_u$ are both progressively measurable. We equip $\mathcal{A}^2_H$ with the norm
\begin{align*}
\Vert u \Vert^2_{\mathcal{A}^2_H}= \E \left(\Vert u(0)\Vert^2_H +\int_0^T \Vert \tilde{u}(t) \Vert^2_{H} \, dt+ \int_0^T  \Vert F_u(t)\Vert^2_{H} \, dt  \right), 
\end{align*}
so that it becomes a Hilbert space. Indeed, the following correspondence 
\begin{align}\label{ito-correspond}
\nonumber (x_0,y,Z)& \in L^2(\Omega; H)\times L^2(\Omega_T; H) \times L^2(\Omega_T; H) \\
\nonumber &\mapsto x_0 + \int_0^t y(s) ds+ \int_0^t Z(s) dW(s) \in \mathcal{A}^2_H,\\
\quad\\
\nonumber u \in \mathcal{A}^2_H &\mapsto (u(0), \tilde u, F_u) \in L^2(\Omega; H)\times L^2(\Omega_T; H) \times L^2(\Omega_T; H), 
\end{align}
induces an isometry, since 
It\^o's formula applied to two processes $u,v \in \mathcal{A}^2_H$ yields
\begin{align*}
\Vr u(t)-v(t)\Vr_H^2=\Vr u(0)& -v(0)\Vr_H^2+2 \int_0^t \langle \tilde u(s)-\tilde v(s), u(s)-v(s) \rangle_H ds \\
&+ \int_0^t \Vr F_u(s)-F_v(s) \Vr^2_H ds+ 2 \int_0^t \langle u(s)-v(s), F_u(s)-F_v(s) \rangle_H dW_s,
\end{align*}
which means that $u=v$ if and only if $u(0)=v(0)$, $F_u=F_v$ and $\tilde u=\tilde v$. We therefore can and shall identify the It\^o space  $\mathcal{A}^2_H$ with the product space $L^2(\Omega;H) \times L^2(\Omega_T;H) \times L^2(\Omega_T;H)$. \\
 The dual space $(\mathcal{A}^2_H)^*$ can also be identified with $L^2(\Omega; H) \times L^2(\Omega_T;H) \times  L^2(\Omega_T; H)$. In other words, each $p \in (\mathcal{A}^2_H)^*$ can be represented by the triplet 
$$p = (p_0, p_1(t), P(t)) \in L^2(\Omega; H) \times L^2(\Omega_T; H) \times  L^2(\Omega_T; H),$$ 
in such a way that the duality can be written as: 
\begin{equation}\label{duality_A22}
\langle u, p \rangle_{{\mathcal{A}^2_H}\times (\mathcal{A}^2_H)^*} =\E   \Big\{ \langle p_0,u(0) \rangle_H + \int_0^T\langle p_1(t),\tilde{u}(t) \rangle_H\, dt + \frac{1}{2} \int_0^T\langle P(t),F_u(t) \rangle_{H} \, dt \Big \}.
\end{equation}
  \subsection{Bolza duality for random processes}
 We now prove the following stochastic analogue of the Bolza duality established in the deterministic case by Rockafellar \cite{rock2}. 

\begin{thm}\label{partial_sd_thm}
Let $(\Omega, \F, \F_t, \P)$ be a  complete probability space with normal filtration, and let $L$ and $M$ be two $\Omega_T$-dependent self-dual Lagrangians on $H \times H$, 
Assume $\ell$ is an $\Omega$-dependent function on $H \times H$, such that $\P$-a.s.  
\begin{equation}\label{boundary-ell}
\ell(\omega,a,b)=\ell^*(\omega,-a,b), \qquad (a,b) \in H\times H.
\end{equation}
The Lagrangian on $\mathcal{A}^2_H\times (\mathcal{A}^2_H)^*$ defined by
\begin{align}\label{partial_sdLag_onA2}
\begin{split}
\mathcal{L}(u,p)&=
\E \, \Big\{\int_0^T L(u(t)-p_1 (t), -\tilde{u}(t)) \, dt+ \ell(u(0)-p_0,u(T)) \\ &\quad + \ \frac{1}{2} \int_0^T M(F_u(t) -P(t),-F_u(t)) \, dt \Big\},
\end{split}
\end{align}
is then partially self-dual. Actually, it is self-dual on the subset $\mathcal{A}^2_H\times {\cal D}$ of $\mathcal{A}^2_H\times (\mathcal{A}^2_H)^*,$ where ${\cal D}:=(\{0\} \times L^2_H \times L^2_H)$.  
\end{thm}
\begin{proof}
Take $(q,v)\in (\mathcal{A}^2_H)^*\times \mathcal{A}^2_H$ with $q$ an element in the dual space identified with the triple $(0, q_1(t), Q(t))$, then
\begin{align*}
\begin{split}
\mathcal{L}^*(q,v)&= \sup_{\substack{u\in \mathcal{A}^2_H \\ \\ p\in (\mathcal{A}^2_H)^*}} \left\{\langle q,u \rangle + \langle v,p\rangle- \mathcal{L}(u,p) \right\}\\
&= \sup_{u\in \mathcal{A}^2_H} \sup_{\substack{p_0\in L^2_H(\Omega) \\ p_1 \in L^2_H(\Omega_T)}} \sup_{P \in L^2_H(\Omega_T)} \E \ \bigg\{\langle p_0, v(0)\rangle + \int_0^T \Big(\langle q_1(t), \tilde{u}(t)\rangle + \langle p_1(t), \tilde{v}(t)\rangle\Big) \, dt \\
& \hspace{3.3cm}+ \, \frac{1}{2} \ \int_0^T \Big(\langle Q(t), F_u(t)\rangle + \langle P(t), G_v(t)\rangle \Big)\, dt \\
& \hspace{3.3cm}- \int_0^T L(u(t)-p_1(t), -\tilde{u}(t)) \, dt -\ell(u(0)-p_0,u(T))\\
& \hspace{3.3cm} - \frac{1}{2} \int_0^T M(F_u(t) -P(t),-F_u(t)) \, dt \bigg\}.
\end{split}
\end{align*}
Make the following  substitutions:
\begin{align*}
u(t)-p_1(t)&=y(t)\in L^2_{H}(\Omega_T)\\
u(0)-p_0&=a\in L^2_H(\Omega)\\
F_u(t)-P(t)&=J(t)\in L^2_H(\Omega_T),
\end{align*}
to obtain
\begin{align*}
\mathcal{L}^*(q,v)= &\sup_{u\in \mathcal{A}^2_H}  \sup_{a\in L^2_H(\Omega)} \sup_{y \in L^2_H(\Omega_T)} \sup_{J \in L^2_H(\Omega_T)} \E \ \Big\{ \langle u(0)-a, v(0)\rangle -\ell(a,u(T)) \\
& \hspace{1cm}+ \int_0^T \Big(\langle q_1(t), \tilde{u}(t)\rangle + \langle u(t)-y(t), \tilde{v}(t)\rangle - L(y(t), -\tilde{u}(t))\Big) \, dt \\
&  + \ \frac{1}{2} \int_0^T \langle Q(t), F_u(t)\rangle + \langle F_u(t)-J(t), G_v(t)\rangle -  M(J(t),-F_u(t))  \, dt \Big\}.
\end{align*}
\normalsize
Use It\^o's formula  (\ref{int_by_parts}) for the processes $u$ and $v$ in $\mathcal{A}^2_H$,
 to get
\begin{align*}
\mathcal{L}^*(q,v)= &\sup_{u\in \mathcal{A}^2_H}  \sup_{a\in L^2_H(\Omega)} \sup_{y \in L^2_H(\Omega_T)} \sup_{J \in L^2_H(\Omega_T)} \E \ \Big\{ \langle a, -v(0)\rangle +\langle u(T),v(T)\rangle -\ell(a,u(T)) \\
& \hspace{1.5cm}+ \int_0^T \langle v(t)-q_1(t), -\tilde{u}(t)\rangle + \langle y(t), -\tilde{v}(t)\rangle - L(y(t), -\tilde{u}(t)) \, dt \\
&  + \frac{1}{2} \int_0^T \langle G_v(t)-Q(t), -F_u(t)\rangle + \langle J(t), -G_v(t)\rangle -  M(J(t),-F_u(t))  \, dt \Big\}.
\end{align*}
In view of the correspondence
\begin{align*}
\nonumber (b,r,Z)& \in L^2(\Omega; H)\times L^2(\Omega_T; H) \times L^2(\Omega_T; H) \\
\nonumber &\mapsto b + \int_0^t r(s) ds+ \int_0^t Z(s) dW(s) \in \mathcal{A}^2_H.\\
\quad\\
\nonumber u \in \mathcal{A}^2_H &\mapsto (u(T), -\tilde u, -F_u) \in L^2(\Omega; H)\times L^2(\Omega_T; H) \times L^2(\Omega_T; H),
\end{align*}
it follows that 
\begin{align*}
\mathcal{L}^*(q,v)&= \sup_{(a,b)\in L^2_H(\Omega) \times L^2_H(\Omega)} \E \ \Big\{ \langle a, -v(0)\rangle +\langle b,v(T)\rangle -\ell(a,b) \Big\} \\
&+ \ \sup_{(y,r) \in L^2_H(\Omega_T) \times L^2_H(\Omega_T)} \E \ \Big\{ \ \int_0^T \langle v(t)-q_1(t),r(t)\rangle + \langle y(t), -\tilde{v}(t)\rangle - L(y(t), r(t)) \, dt \Big\}\\
&+ \frac{1}{2} \sup_{\substack{J\in L^2_ H(\Omega_T) \\ Z\in L^2_ H(\Omega_T)}} \E \ \Big\{\int_0^T \langle G_v(t)-Q(t), Z(t)\rangle + \langle J(t), -G_v(t)\rangle -  M(J(t),Z(t))\, dt \Big\},
\end{align*}
and therefore, 
\begin{align*}
\mathcal{L}^*(q,v)&=   \E \ \ell^*(-v(0),v(T)) + \E  \int_0^T  L^*(-\tilde{v}(t), v(t)-q_1(t)) \, dt \\
&\quad + \frac{1}{2}\, \E \int_0^T  M^*(-G_v(t),G_v(t)-Q(t)) \, dt.
\end{align*}
Now with the self-duality assumptions on $L$ and $M$, and the condition on $\ell$, we have
$\mathcal{L}^*(0,v)= \mathcal{L}(v,0),$ for every $v\in \mathcal{A}^2_H$, 
which means that $\mathcal{L}$ is partially self-dual on $\mathcal{A}^2_H \times (\mathcal{A}^2_H)^*$.
\end{proof}

\section{Variational resolution of stochastic equations driven by additive noise}\label{additive}
For simplicity, we shall work in an $L^2$-setting in $w$ and in time.  
 \subsection{A variational principle on It\^o space}
The following  is now a direct consequence of Theorem \ref{partial_sd_thm} and Theorem \ref{var_prin}.
\begin{prop}\label{main.var}
Let $(\Omega, \F, \F_t, \P)$ be a complete probability space with normal filtration and let $H$ be a Hilbert space. Suppose $L$ and $M$ are $\Omega_T$-dependent self-dual Lagrangians on $H \times H$, and $\ell$ is an $\Omega$-dependent time-boundary Lagrangian on $H \times H$.
 Assume that for some positive $C_1, C_2$ and $C_3$, we have
\begin{align}\label{L-cond}
\begin{split}
\E \int_0^T L(t, v(t),0) \ dt &\leq C_1(1+\Vr v\Vr^2_{L^2_H(\Omega_T)}) \qquad \ntx{for} \  v \in L^2_H(\Omega_T),\\
\E \ \ell(a,0)&\leq C_2(1+\Vr a \Vr^2_{L^2_H(\Omega)}) \qquad \ \ntx{for} \ a \in L^2_H(\Omega),\\
\E \int_0^T M(\sigma(t),0) \ dt &\leq C_3(1+\Vr \sigma \Vr^2_{L^2_H(\Omega_T)}) \qquad \ntx{for} \ \sigma \in L^2_H(\Omega_T).
\end{split}
\end{align}
Consider on  $\mathcal{A}^2_H$ the functional
\begin{align*}
{I}(u)=
\E \, \Big\{\int_0^T &L(t,u(t), -\tilde{u}(t)) \, dt+ \ell(u(0),u(T))+ \frac{1}{2} \int_0^T M(F_u(t),-F_u(t)) \, dt \Big\}.
\end{align*} 
Then, there exists $v \in \mathcal{A}^2_H$ such that 
$I(v)=\underset{u\in \mathcal{A}^2_H}\inf I(u)=0,$ and consequently, $\P$-a.s. and for almost all $t\in [0,T]$, we have
\begin{equation}\label{minimize}
-\tilde{v}(t) \in \bar{\partial} L(t,v(t))\\
\end{equation}
$$ (-v(0),v(T))\in \partial \ell(v(0),v(T))$$
$$-F_v(t)\in \bar{\partial}M(F_v(t)).$$
Moreover,  if $L$ is strictly convex, then $v$ is unique.
 \end{prop}
\begin{proof} 
The functional  $I$ can be written as $I(u)=\msL(u,0)$, where $\msL$ is the partially self-dual Lagrangian defined by (\ref{partial_sdLag_onA2}).
In order to apply Theorem \ref{var_prin}, we need to verify the coercivity condition. To this end, we use Conditions (\ref{L-cond}) to show that the map $p\rightarrow \msL(0,p)$ is bounded on the bounded sets of $(\mathcal{A}^2_H)^*$. Indeed, 
\begin{align*}
\msL(0,p)&= \E \, \Big\{\int_0^T L(t, p_1(t), 0) \, dt+ \ell(-p_0,0) + 
\frac{1}{2} \int_0^T M(-P(t),0) \, dt \Big\}\\
&\leq C \, \Big(3+\Vr p_1\Vr^2_{L^2_H(\Omega_T)}+\Vr p_0 \Vr^2_{L^2_H(\Omega)}+\Vr P \Vr^2_{L^2_H(\Omega_T)}\Big),
\end{align*}
and by duality, $\underset{\Vr u \Vr \rightarrow \infty} \lim \dfrac{\msL(u,0)}{\Vr u \Vr}= +\infty$.
By Theorem \ref{var_prin}, there exists $v\in \mathcal{A}^2_H$ such that $I(v)=0$. We now rewrite $I$ as follows:
\begin{align*}
0={I}(v)=
\E \ \Big\{\int_0^T& L(t,v(t), -\tilde{v}(t))+\langle v(t), \tilde{v}(t) \rangle \, dt - \int_0^T \langle v(t), \tilde{v}(t) \rangle \, dt \\
&+\ell(v(0),v(T)) + 
\frac{1}{2} \int_0^T M(F_v(t),-F_v(t)) \, dt \Big\}.
\end{align*}
\noindent By It\^o's formula
\begin{align*}\label{itoo}
\E \int_0^T \langle v(t), \tilde{v}(t) \rangle = \frac{1}{2} \, \E \Vr v(T)\Vr^2_H - \frac{1}{2} \, \E \Vr v(0)\Vr^2_H - \frac{1}{2} \,\E \int_0^T \Vr F_v(t)\Vr^2_H \, dt,
\end{align*}\label{itoo}
which yields
\begin{align*}
0 \ = \ {I}(v)=& \ 
\E \ \Big\{\ \int_0^T \Big(L(t,v(t), -\tilde{v}(t))+\langle v(t), \tilde{v}(t) \rangle\Big) \ dt \Big\} \\
&+ \E \, \Big\{\ell(v(0),v(T)) -\frac{1}{2} \Vr v(T)\Vr^2_H +\frac{1}{2} \Vr v(0)\Vr^2_H\Big\} \\
&+\frac{1}{2} \, \E \, \Big\{\int_0^T \Big(\Vr F_v\Vr^2_H+ M(F_v(t),-F_v(t))\Big)\, dt  \Big\}.
\end{align*}
The self-duality of the Lagrangians $L$ and $M$ and the hypothesis on the boundary Lagrangian, yield that for a.e. $t\in [0,T]$ and $\mathbb{P}$-a.s. each of the integrands inside the curly-brackets are non-negative, thus
$$L(t,v(t), -\tilde{v}(t))+\langle v(t), \tilde{v}(t) \rangle=0,$$
$$\ell(v(0),v(T)) -\frac{1}{2} \Vr v(T)\Vr^2_H +\frac{1}{2} \Vr v(0)\Vr^2_H=0,$$
$$M(F_v(t),-F_v(t))+\langle F_v, F_v\rangle=0 ,$$
which translate into the three assertions in (\ref{minimize}).\\
Finally, if $L$ is strictly convex, then the functional $I$ is strictly convex and the minimum is attained uniquely.
\end{proof}
\subsection{Regularization via inf-involution}
The boundedness condition 
(\ref{L-cond}) is quite restrictive and not satisfied by most Lagrangians of interest. One way to deal with such a difficulty is to assume similar bounds on $L$ but in stronger Banach norms. Moreover, we need to find more regular solutions that are valued in more suitable Banach spaces than $H$.   
To this end, we consider an evolution triple $V\subset H \subset V^*$, where $V$ is a  
 reflexive Banach space and $V^*$ is its dual. We recall the following easy lemma from \cite{G-book}.
 \begin{lem}\label{boundedness-L}
Let $L$ be a self-dual Lagrangian on $V \times V^*$.
\begin{enumerate}
\item If for some $r>1$ and $C>0$, we have
$L(u,0)\leq C(1+\Vr u\Vr_V^{r})$ for all $u\in V,$
then there exists $D>0$ such that
$L(u,p) \geq D(\Vr p\Vr_{V^*}^{s}-1)$ for all  $(u,p)\in V\times V^*$, where  $\frac{1}{r}+\frac{1}{s}=1.$
\item If for $C_1, C_2>0$ and $r_1\geq r_2>1$, we have  
$$C_2(\Vr u\Vr_V^{r_2}-1)\leq L(u,0) \leq C_1(1+\Vr u\Vr_V^{r_1}) \qquad \ntx{for all} \ u\in V,$$
then, there exists $D_1, D_2>0$ such that
\begin{equation}
D_2(\Vr p\Vr_{V^*}^{s_1}+\Vr u\Vr_V^{r_2}-1)\leq L(u,p) \leq D_1(1+\Vr u\Vr_V^{r_1}+\Vr p\Vr_{V^*}^{s_2}).
\end{equation}
where $\frac{1}{r_i}+\frac{1}{s_i}=1$ for $i=1,2$, and therefore $L$ is continuous in both variables.
\end{enumerate}
\end{lem}
\begin{prop}\label{triple-thm}
Consider a Gelfand triple $V\subset H\subset V^*$ and let $L$ be an $\Omega_T$-dependent self-dual Lagrangian on $V \times V^*$. Let $M$ be an $\Omega_T$-dependent self-dual Lagrangian on $H\times H$, and $\ell$ an $\Omega$-dependent boundary Lagrangian on $H \times H$ satisfying $\ell^*(a,b)=\ell(-a,b)$. Assume the following conditions hold:
\begin{description}
\item [($A_1$)] For some $m,n >1$, $C_1,C_2>0$, 
\begin{equation*}
C_2(\Vr v\Vr^m_{L^2_V(\Omega_T)}-1)\leq \E \int_0^T L(t, v(t),0) \ dt \leq C_1(1+\Vr v\Vr^n_{L^2_V(\Omega_T)})\quad \ntx{for all} 
\ v \in L^2(\Omega_T;V).
\end{equation*}
\item[($A_2$)] For some $C_3>0$, 
\begin{equation*}
\E \ \ell(a,b)\leq C_3(1+\Vr a \Vr^2_{L^2_H(\Omega)}+\Vr b \Vr^2_{L^2_H(\Omega)})\quad \ntx{for all} 
\ a,b \in L^2(\Omega;H).
\end{equation*}
\item[($A_3$)] For some $C_4>0$, 
\begin{equation*}
\E \int_0^T M(G_1(t),G_2(t)) dt \leq C_4(1+\Vr G_1 \Vr^2_{L^2_H(\Omega_T)}+\Vr G_2 \Vr^2_{L^2_H(\Omega_T)}) \quad\ntx{for all} 
\ G_1,G_2 \in L^2_H(\Omega_T).
\end{equation*}
\end{description}
Then, there exists $v\in \mathcal{A}_H^2$ with trajectories in $L^2(\Omega_T;V)$ such that $\tilde{v}\in L^2(\Omega_T;V^*)$, at which the minimum of the following functional is attained and is equal to $0$.
\begin{align*}
{I}(u)=
\E \, \Big\{\int_0^T &L(t,u(t), -\tilde{u}(t) \, dt+ \ell(u(0),u(T))
+ \frac{1}{2} \int_0^T M(F_u(t),-F_u(t)) \, dt \Big\}.
\end{align*}
 Consequently, $\P$-a.s. and for almost all $t\in [0,T]$, we have
\begin{equation}\label{minimize1}
-\tilde{v}(t) \in \bar{\partial} L(t,v(t))\\
\end{equation}
$$ (-v(0),v(T))\in \partial \ell(v(0),v(T))$$
$$-F_v(t)\in \bar{\partial}M(F_v(t)).$$
\end{prop}

\begin{proof}
First, apply Lemma \ref{lift-to-H} to lift $L$ to an $\Omega_T$-dependent self-dual Lagrangian on $H\times H$, then consider for $t\in [0,T]$ and $\P$-a.s., the $\lambda$-regularization of $L$, that is 
$$L_\lambda(u,p)=\inf_{z\in H}\Big\{L(z,p)+\frac{\Vr u-z\Vr^2_H}{2\lambda}+\frac{\lambda}{2}\Vr p\Vr^2_H \Big\}.$$
By Lemma \ref{conv}, $L_\lambda$ is also an $\Omega_T$-dependent self-dual Lagrangian on $H\times H$ in such a way that the conditions (\ref{L-cond}) of Proposition \ref{main.var} hold. Hence, there exists $v_\lambda \in \mathcal{A}_H^2$ such that 
\begin{align*}
0=\E \ \Big\{\int_0^T L_\lambda(v_\lambda(t), -\tilde{v}_\lambda(t))\, dt +\ell(v_\lambda(0),v_\lambda(T)) + 
\frac{1}{2} \int_0^T M(F_{v_\lambda}(t),-F_{v_\lambda}(t)) \, dt \Big\}.
\end{align*}
Since $L$ is convex and lower semi-continuous, then $dt\otimes \P$ a.s, there exists $J_\lambda(v_\lambda) \in H$ so that
$$L_\lambda(v_\lambda(t),-\tilde{v}_\lambda(t))=L(J_\lambda(v_\lambda)(t),-\tilde{v}_\lambda(t))+\frac{\Vr v_\lambda(t)-J_\lambda(v_\lambda)(t)\Vr^2_H}{2\lambda}+\frac{\lambda}{2}\Vr \tilde{v}_\lambda(t)\Vr^2_H,$$
and hence
\begin{align}\label{lambda-inf0}
\nonumber 0=\E \ \Big\{\int_0^T &\Big(L(J_\lambda(v_\lambda)(t),-\tilde{v}_\lambda(t))+\frac{\Vr v_\lambda(t)-J_\lambda(v_\lambda)(t)\Vr_H^2}{2\lambda}+\frac{\lambda}{2}\Vr \tilde{v}_\lambda(t)\Vr^2_H\Big)\, dt \\
&+\ell(v_\lambda(0),v_\lambda(T)) + 
\frac{1}{2} \int_0^T M(F_{v_\lambda}(t),-F_{v_\lambda}(t)) \, dt \Big\}.
\end{align}
From (\ref{lambda-inf0}), condition ($A_1$) and the assertion of part (2) of Lemma \ref{boundedness-L}, we can deduce that $J_\lambda(v_\lambda)$ is bounded in $L^2(\Omega_T;V)$ and $\tilde{v}_\lambda$ is bounded in $L^2(\Omega_T;V^*)$. Also from condition ($A_2$) and ($A_3$), we can deduce the following estimates:
\[
\E\, \int_0^T M(G, H)\, dt \geq C(\|G\|^2_{L^2_H(\Omega_T)}-1)\quad {\rm and} \quad \E \, \ell (a, b) \geq C(\|b\|^2_{L^2_H(\Omega)}-1).
\]
These coercivity properties, together with (\ref{lambda-inf0}), imply that $v_\lambda(0)$ and $v_\lambda(T)$ are bounded in $L^2(\Omega;H)$, and that  $F_{v_\lambda}$ is bounded in $L^2(\Omega_T;H)$. Moreover, since all other terms in (\ref{lambda-inf0}) are bounded below, it follows that 
$$\E \int_0^T \Vr v_\lambda(t)-J_\lambda(v_\lambda)(t)\Vr^2 dt \leq 2\lambda C \quad \ntx{for some} \ C>0.$$
Hence $v_\lambda$ is bounded in $\mathcal{A}^2_H$ and there exists a subsequence $v_{\lambda_j}$ that converges weakly to a path $v\in L^2(\Omega_T;V)$ such that $\tilde{v} \in L^2(\Omega_T;V^*)$, and 
$$J_{\lambda_j}(v_{\lambda_j}) \rightharpoonup v \quad \ntx{in} \quad L^2(\Omega_T;V)$$
$$\tilde{v}_{\lambda_j} \rightharpoonup \tilde{v} \quad \ntx{in} \quad L^2(\Omega_T;V^*)$$
$$v_{\lambda_j} \rightharpoonup v \quad \ntx{in} \quad L^2(\Omega_T;H)$$
$$v_{\lambda_j}(0) \rightharpoonup v(0), \quad v_\lambda(T) \rightharpoonup v(T) \quad \ntx{in} \quad L^2(\Omega;H)$$
$$F_{v_{\lambda_j}} \rightharpoonup F_v \quad \ntx{in} \quad L^2(\Omega_T;H).$$
Since $L, \ell$ and $M$ are lower semi-continuous, we have
\begin{align*}
I(v) &\leq \liminf_{j} \E \ \Big\{\int_0^T \Big(L(J_{\lambda_j}(v_{\lambda_j})(t),-\tilde{v}_{\lambda_j}(t))+\frac{\Vr v_{\lambda_j}(t)-J_{\lambda_j}(v_{\lambda_j})(t)\Vr^2}{2{\lambda_j}}+\frac{{\lambda_j}}{2}\Vr \tilde{v}_{\lambda_j}(t)\Vr^2\Big)\, dt \\
& \qquad \qquad \qquad+\ell(v_{\lambda_j}(0),v_{\lambda_j}(T)) + 
\frac{1}{2} \int_0^T M(F_{v_{\lambda_j}}(t),-F_{v_{\lambda_j}}(t)) \, dt \Big\}=0.
\end{align*}
For the reverse inequality, we use the self-duality of $L$ and $M$ and the fact that $\ell(-a,b)=\ell^*(a,b)$ to deduce that
\begin{align*}
{I}(v)=& \ 
\E \ \Big\{\ \int_0^T \Big(L(v(t), -\tilde{v}(t))+\langle v(t), \tilde{v}(t) \rangle\Big) \ dt \Big\} \\
&+ \E \, \Big\{\ell(v(0),v(T)) -\frac{1}{2} \Vr v(T)\Vr^2_H +\frac{1}{2} \Vr v(0)\Vr^2_H\Big\} \\
&+\frac{1}{2} \, \E \, \Big\{\int_0^T \Big(\Vr F_v\Vr^2_H+ M(F_v(t),-F_v(t))\Big)\, dt  \Big\}\geq 0.
\end{align*}
Therefore, $I(v)=0$ and the rest of the proof is similar to the last part of the proof in Proposition \ref{main.var}.
\end{proof}
We now deduce the following.
\begin{thm} \label{one-sol} Consider a Gelfand triple $V\subset H\subset V^*$, and let $A: D(A)\subset V\to V^*$ be an $\Omega_T$-dependent progressively measurable maximal monotone operator satisfying 
\begin{align*}\label{A-cond2}
 \langle A_{w,t}u,u\rangle \geq \max \{  c_1(\omega,t)\Vr u \Vr_V^\alpha -m_1(\omega,t), \, c_2(\omega,t)\Vr Au \Vr^\beta_{V^*}-m_2(\omega,t)\}, 
\end{align*}
where $c_1, c_2 \in L^\infty(\Omega_T, dt \otimes \P)$ and $m_1, m_2 \in L^1(\Omega_T, dt \otimes \P)$.
  Let $B$ be a given $H$-valued progressively measurable process in $L^2(\Omega_T;H)$, and $u_0$ a given random variable in $L^2(\Omega,\F_0,\mathbb{P};H)$.  Then, the equation  
\begin{equation}\label{basic.2}
\begin{cases}
du(t)= -A (t,u(t)) dt + B(t) dW(t)\\
u(0)=u_0,
\end{cases}
\end{equation}
has a solution $u\in {\cal A}^2_H$ that is valued in $V$. It can be obtained by minimizing the functional 
\begin{align*}
I(u)& =\  \E \int_0^T L(u(t), -\tilde{u}(t))\, dt \\
& \quad+ \E\ \Big(\frac{1}{2}  \Vr u(0)\Vr^2_H+ \frac{1}{2}  \Vr u(T)\Vr^2_H -2\langle u_0, u(0)\rangle_H +\Vr u_0\Vr^2_H\Big) \\
&\quad + \E \int_0^T\Big(\frac{1}{2} \, \Vr F_u(t)-2B(t)\Vr_H^2\,+ \,\frac{1}{2} \, \Vr F_u(t) \Vr_H^2 -2 \langle F_u(t),B(t) \rangle_H \Big) \, dt,
\end{align*}
where $L$ is a self-dual Lagrangian such that $\bar \partial L(t, \cdot)=A(t, \cdot)$, $\P$-almost surely.
\end{thm}
\begin{proof}
It suffices to apply Proposition \ref{triple-thm} with the self-dual Lagrangian $L$ associated with $A$, the time boundary $\Omega$-dependent Lagrangian $\ell_{u_0}$ on $H\times H$ given by 
$$\ell_{u_0}(a,b)=\frac{1}{2}  \Vr a\Vr^{^2}_H+ \frac{1}{2}  \Vr b\Vr^{^2}_H-2\langle u_0(w), a\rangle_H+\Vr u_0(w)\Vr^{^2}_H,$$
and the $\Omega_T$-dependent self-dual Lagrangian $M$ on $L_H^2(\Omega_T)$, given by
$$M_B(G_1,G_2)=\Psi_{B(w, t)}(G_1)+\Psi_{B(w,t)}^*(G_2),$$ 
where $\Psi_{B(w, t)}:H \rightarrow \pR$ is the convex function  
$\Psi_{B(w, t)} (G)=\frac{1}{2} \Vr G-2B(w, t)\Vr^2_H.$
\end{proof}


 \section{Applications to various SPDEs with additive noise\label{Ex}}

In the following examples, we shall assume $D$ is a smooth bounded domain in $\R^n$, $W$ is a real Brownian motion, and $B:\Omega \times [0,T] \rightarrow L^2(D)$ is a fixed progressively measurable stochastic process.

\subsection{Stochastic evolution driven by diffusion and transport}

\noindent Consider the following stochastic transport equation:
\begin{equation}\label{transport_stoch}
\begin{cases}
du = (\Delta u  + \textbf{a}(x)\cdot \nabla u) dt+ B(t) dW \ &\textnormal{on} \ [0,T] \times  D \\
u(t, x)=0 &\textnormal{on} \ [0,T] \times  \partial D \\
u(0, x)=u_0(x) & \textnormal{on} \ D,
\end{cases}
\end{equation} 
where  
$\textbf{a}: D \rightarrow \R^n$ is a smooth vector field with compact support in $D$, such that $\div(\textbf{a}) \geq 0$. Assume $u_0 \in L^2(\Omega,\F_0,\mathbb{P}; H^1_0(D))$ such that $\Delta u_0 \in L^2(D)$, $\mathbb{P}$-a.s. \\
Consider the operator $\Gamma u=\textbf{a} \cdot \nabla u+\frac 1 2 (\div\, \textbf{a})\, u,$ which, by Green's formula, is skew-adjoint on $H^1_0(D)$. Also consider the convex function 
$$\varphi(u)=\begin{cases}
\frac 1 2 \int_D \vr \nabla u \vr^2 dx +\frac 1 4 \int_D (\div\, \textbf{a})\vr u\vr^2 dx \qquad &u \in H^1_0(D)\\
+\infty & \textnormal{otherwise,}
\end{cases}$$
which is clearly coercive on $H^1_0(D)$. Consider the Gelfand triple
$
H^1_0(D) \subset L^2(D) \subset H^{-1}(D)$, and the self-dual Lagrangian on $H^1_0(D)\times H^{-1}(D)$, defined by
\[
L(u,p)=\varphi (u) +\varphi^*(\Gamma u +p). 
\]
The corresponding functional on It\^o space is then, 
\begin{align*}
I(u)& =\  \E \left\{ \int_0^T  \left( \Big(\frac 1 2 \int_D (\vr \nabla u \vr^2 dx +\frac 1 4 \int_D (\div\, \textbf{a})\vr u\vr^2 ) dx \Big)+ \varphi^*\Big(-\tilde{u}(t, \cdot )+\Gamma (u (t, \cdot))\Big) \right) \, dt \right\} \\
 & \quad +
\E \, \left\{\frac{1}{2} \int_0^T \left( \int_D\Big(|F_u(t, x)|^2+ 2|B(t, x)|^2 -4  F_u(t,x)B(t,x)\Big) dx\right)  \, dt \right\}\\
 &\quad +  \E \, \left\{\int_D \Big(\frac{1}{2}|u(0, x)|^2 + \frac{1}{2}|u(T, x)|^2 - 2u_0(x)u(0, x)+ \frac{1}{2} |u_0(x)|^2\Big)\, dx \right\}.
\end{align*} 
Apply Theorem \ref{one-sol} to find a path $v \in \mathcal{A}^2_{L^2(D)}$, valued in $H^1_0(\Omega)$, that minimizes $I$ in such a way that $I(v)=0$, to obtain
 $$-\tilde{v}+\textbf{a} \cdot \nabla v+ \frac 1 2(\div\, \textbf{a})v \in \partial \varphi(v)= -\Delta v + \frac 1 2(\div\, \textbf{a})v,$$
$$v(0)=u_0, \ F_v=B.$$
The process 
$v(t)= u_0 +\int_0^t \Delta v(s) ds + \int_0^t \textbf{a} \cdot \nabla v(s) ds+\int_0^t B(s) dW(s)$
is therefore a solution to (\ref{transport_stoch}).

\subsection{Stochastic porous media}

\noindent Consider the following SPDE,
\begin{align}\label{porous}
\begin{cases}
du(t)=\Delta u^p(t)dt+B(t)dW(t) \qquad &\ntx{on} \ D\times [0,T]\\
u(t, x)=0 &\textnormal{on} \ [0,T] \times  \partial D \\
u(0, x)=u_0(x) &\ntx{on} \ D,
\end{cases}
\end{align}
where $p\geq \frac{n-2}{n+2}$, and $u_0 \in L^2(\Omega,\F_0,\mathbb{P}; H^{-1}(D))$.\\
Equip the Hilbert space $H=H^{-1}(D)$ with the inner product 
$$\langle u,v \rangle_{H^{-1}}=\langle u,(-\Delta)^{-1}v \rangle
=\int_D u(x) (-\Delta)^{-1}v(x) \, dx.$$
Since $p\geq \frac{n-2}{n+2}$, 
$L^{p+1}(D) \subset H^{-1}(D) \subset L^{\frac{p+1}{p}}(D)$
is an evolution triple. 
Now consider the convex functional
\begin{equation*}
\varphi(u)=\begin{cases}
\frac{1}{p+1} \int_D \vr u(x)\vr^{p+1} dx \qquad &\ntx{on} \ L^{p+1}(D)\\
+\infty  &\ntx{elsewhere},
\end{cases}
\end{equation*}
whose Legendre conjugate is given by
$$\varphi^*(u^*)=\frac{p}{p+1}\int_D \vr (-\Delta)^{-1} u^* \vr^{\frac{p+1}{p}}dx.$$
Now, minimize the following self-dual  functional on $\mathcal{A}^2_H$,
 \begin{align*}
I(u)=&\E \ \left\{ \frac{1}{p+1} \int_0^T \int_D \Big(\left\vr u(x)\right\vr^{p+1} + p\, \left\vr (-\Delta)^{-1}(-\tilde{u}(t)) \right\vr^{\frac{p+1}{p}} \Big)dx \, dt \right\} \\
&\hspace{0.9cm}+\E\ \left\{\frac{1}{2}  \Vr u(0)\Vr^2_{_{H^{-1}}}+ \frac{1}{2}  \Vr u(T)\Vr^2_{_{H^{-1}}} +\Vr u_0\Vr^2_{_{H^{-1}}}-2\langle u_0,u(0, \cdot) \rangle_{_{H^{-1}}} \right\}\\
&\hspace{0.9cm} + \E \left\{ \int_0^T\Big(\frac{1}{2}\Big(\Vr F_u(t)\Vr_{_{H^{-1}}}^2\,+ 2\Vr B(t)\Vr_{_{H^{-1}}}^2 -4 \langle F_u(t),B(t) \rangle_{_{H^{-1}}}  \Big) dt\right\}.
\end{align*} 
Apply Theorem \ref{one-sol} to find a process $v\in \mathcal{A}^2_H$ with values in $L^{p+1}(D)$ such that 
$$(-\Delta)^{-1}(-\tilde{v}(t)) \in \partial \varphi (v(t))= v^p, F_v=B,\, \hbox{and $v(0)=u_0$.}
$$ 
It follows that $v(t)= u_0 +\int_0^t \Delta v^p(s) ds +\int_0^t B(s) dW(s), $
provides a solution for  (\ref{porous}).

\subsection{Stochastic PDE involving the p-Laplacian}

\noindent Consider the  equation
\begin{equation*}
\begin{cases}
du=(\Delta_p u - u\vr u \vr^{p-2}) dt+ B(t)dW \qquad &\on \ D\times [0,T]\\
u(t, x)=0 &\textnormal{on} \ [0,T] \times  \partial D \\
u(0, x)=u_0(x) & \on\ D,
\end{cases}
\end{equation*}
where $p\in [2, +\infty)$,  $\Delta_p u= \div(\vr \nabla u \vr^{p-2}\nabla u)$ is the $p$-Laplacian operator, and $u_0$ is given such that $u_0 \in W_0^{1,p}(D) \cap \{u; \Delta_p u \in L^p(D)  \}$. It is clear that $W_0^{1,p}(D) \subset L^p(D)$ continuously and densely, which ensures that the functional
\begin{align*}
\varphi(u)=\frac 1 p \int_D \vr \nabla u(x)\vr^p dx + \frac{1}{p} \int_D  \vr u(x)\vr^p dx,
\end{align*}
is convex, lower semi-coninuous and coercive on $W_0^{1,p}(D)$ with respect to the evolution triple 
$$W_0^{1,p}(D) \subset L^p(D) \subset L^2(D) \subset W_0^{1,p}(D)^* \subset L^q(D),$$
where $\frac{1}{p}+\frac{1}{q}=1$. 
Theorem \ref{one-sol} applies to the self-dual functional 
\begin{align*}
I(u)& =\  \E \int_0^T  \Big( \varphi(t,u)+ \varphi^*(t,-\tilde{u})\Big) \, dt \\
& \quad+ \E\ \Big(\frac{1}{2}  \Vr u(0)\Vr^2_{_{L^2(D)}}+ \frac{1}{2}  \Vr u(T)\Vr^2_{_{L^2(D)}} -2\langle u_0, u(0)\rangle +\Vr u_0\Vr^2_{_{L^2(D)}}\Big) \\
&\quad + \E \int_0^T\Big(\frac{1}{2} \, \Vr F_u(t)\Vr_{_{L^2(D)}}^2\,+  \Vr B(t) \Vr_{_{L^2(D)}}^2 - 2\langle F_u(t),B(t) \rangle \Big) \, dt.
\end{align*} 
to yield a $W_0^{1,p}(D)$-valued process $v\in \mathcal{A}^2_{L^2(D)}$,  where the 
null infimum is attained. It follows that  
$$-\tilde{v} \in \partial \varphi(v)= -\Delta_p v + v\vr v \vr^{p-2},$$
$$v(0)=u_0, \ F_v=B,$$
and hence
$
v(t) - u_0 - \int_0^t B(s) dW(s)=\int_0^t \tilde{v}(s) ds
= \int_0^t \Delta_p v(s) ds - \int_0^t v(s)\vr v(s) \vr^{p-2} ds.
$

\section{Non-additive noise driven by self-dual Lagrangians\label{non-add}}

In this section, we give a variational resolution for stochastic equations of the form
\begin{equation}\label{non-SDE}
\begin{cases}
du = -\bar \partial \mathcal{L}(u)(t)\, dt- \Lambda u (t)\, dt + B(t, u(t)) dW\\
u(0)=u_0, 
\end{cases}
\end{equation}
where ${\mathcal L}$ is a self-dual Lagrangian on $L^\alpha(\Omega_T;V) \times L^\beta(\Omega_T;V^*)$, $1< \alpha <+\infty$ and $\beta$ is its conjugate,
and where $V\subset H\subset V^*$ is a given Gelfand triple. Here $\Lambda: D(\Lambda) \subset L^\alpha(\Omega_T;V) \to L^\beta(\Omega_T;V^*)$ is an appropriate nonlinear operator to be described below.
We shall assume that ${\cal L}$ satisfies the following boundedness conditions:
\begin{align}\label{L10}
 C_2 (\Vr u \Vr^\alpha_{L^\alpha_V(\Omega_T)}+\Vr p \Vr^\beta_{L^\beta_{V^*}(\Omega_T)}-1) \leq  \mathcal{L}(u,p)
\leq C_1(1+\Vr u \Vr^\alpha_{L^\alpha_V(\Omega_T)}+\Vr p \Vr^\beta_{L^\beta_{V^*}(\Omega_T)}).
\end{align}
and 
\begin{equation}\label{L20}
\Vr \bar{\partial}\mathcal{L}(u)\Vr_{L^\beta_{V^*}(\Omega_T)} \leq C_3(1+ \Vr u \Vr^{\alpha-1}_{L^\alpha_V(\Omega_T)}).
\end{equation}
Note that in the last section, we worked in a Hilbertian setting, then used inf-convolution to find a solution that is valued in the Sobolev space $V$. 
This approach does not work in the non-additive case, since we need to work with stronger topologies on the space of It\^o processes that will give the operator $B$ a chance to be completely continuous. We shall therefore strengthen the norm on the It\^o space over a Gelfand triple, at the cost of losing coercivity, that we shall recover through perturbation methods.  \\
More precisely, we are searching for  a solution $u$ of the form
\begin{equation}\label{ito-form}
u(t)=u(0)+ \int_0^t \tilde{u}(s) ds + \int_0^t F_u(s)dW(s),
\end{equation}
where $u \in L^\alpha(\Omega_T;V)$, $\tilde{u} \in L^\beta(\Omega_T;V^*)$ and $F_u \in L^2(\Omega _T;H)$  are progressively measurable. The space of such processes, 
will be denoted by $\mathcal{Y}^{\alpha}_V$, and will be equipped with the norm,
\begin{align*}
\Vert u \Vert_{\mathcal{Y}^{\alpha}_V}= \Vert u(t) \Vert_{L^\alpha_V(\Omega_T)}+ \Vert \tilde{u}(t) \Vert_{L^\beta_{V^*}(\Omega_T)} + \Vert F_u(t)\Vert_{L^2_H(\Omega_T)}.
\end{align*}
As shown in \cite{PR}, any such a process $u\in  \mathcal{Y}^\alpha_V$ has a $dt \otimes \P$-equivalent version $\hat u$ that is a $V$-valued progressively measurable process that satisfies --among other things-- 
 the following  It\^o's formula: 
  $\P$-a.s. and for all $t\in[0,T]$,
\begin{equation*}
\Vr u(t)\Vr_H^2=\Vr u(0)\Vr_H^2+2 \int_0^t \langle \tilde{u}(s), \hat u(s) \rangle_{_{V^*, V}} ds + \int_0^t \Vr F_u(s)\Vr^2_H ds+ 2 \int_0^t \langle u(s), F_u(s)\rangle_H dW(s),
\end{equation*}
In particular, we have for all $t\in[0,T]$, 
\begin{equation}\label{ito}
\E(\Vr u(t) \Vr^2_H)= \E(\Vr u(0) \Vr^2_H)+\E\int_0^t \Big(2\langle\tilde{u}(s), \hat{u}(s) \rangle_{_{V^*, V}}+ \Vr F_u(s) \Vr^2_H \Big)ds.
\end{equation}
Furthermore, we have  $u \in C([0,T];H)$. In fact, one can deduce  that for any $u \in  \mathcal{Y}_V^\alpha$,  $u\in C([0,T];V^*)$ and $u \in L^\infty(0,T;H)$ $\P$-a.s (\cite{Pardoux1} and \cite{PR}).
 From now on, a  process $u$ in $\mathcal{Y}^\alpha_V$ will always be identified with its $dt \otimes \P$-equivalent $V$-valued version $\hat{u}$. 

\begin{thm} \label{last-thm} Consider a self-dual Lagrangian ${\cal L}$ on $L^\alpha(\Omega_T;V) \times L^\beta(\Omega_T;V^*)$ satisfying (\ref{L10}) and (\ref{L20}), and let $B:  \mathcal{Y}^\alpha_V \rightarrow L^2(\Omega_T;H)$ be a --not-necessarily linear-- weak-to-norm continuous map such that for some 
 $C>0$ and $0<\delta< \frac{\alpha+1}{2}$, 
\begin{equation}\label{B-bound}
\Vr Bu \Vr_{L^2_H(\Omega_T)} \leq C \Vr u\Vr^\delta_{L^\alpha_V(\Omega_T)} \quad\hbox{for any $u\in \mathcal{Y}^\alpha_V$.}
\end{equation}
Suppose $\Lambda$ is a weak to weak continuous operator from $\mathcal{Y}^\alpha_V$ into $L^\beta(\Omega_T;V^*)$ such that 
\begin{equation}\label{Lambda-cond.0}
\E \int_0^T \langle \Lambda u(t),u(t) \rangle =0 \quad \hbox{for every $u \in \mathcal{Y}^\alpha_V,$}
\end{equation} and
\begin{equation}\label{Lambda-cond}
\Vr \Lambda u\Vr_{L^\beta_{V^*}(\Omega_T)} \leq 
f\Big(\Vr u\Vr_{L^\alpha_V(\Omega_T)}\Big),
\end{equation} 
where $f$ is a non-decreasing continuous real function.\\ 
Let $u_0$ be a given random variable in $L^2(\Omega, \F_0, \mathbb{P}; H)$.  Equation (\ref{non-SDE}) 
has then a solution $u$ in $ \mathcal{Y}^\alpha_V$, that is a stochastic process satisfying 
\begin{equation}
u(t)=u_0-\int_0^t \left(\bar{\partial}{{\cal L}}(u)(s)+\Lambda u(s) \right)ds + \int_0^t Bu(s)dW(s).
\end{equation}
\end{thm}
We would like to apply Theorem \ref{superpose-thm} to ${\cal L}$  on $L^\alpha(\Omega_T;V)\times L^\beta(\Omega_T;V^*)$ and to the 
following operators acting on $G=\{u \in \mathcal{Y}^\alpha_V; u(0)=u_0\}$,
 \begin{align*}
A_1: G \subset \mathcal{Y}^\alpha_V &\rightarrow L^\alpha(\Omega_T;V), \hspace{2cm}   \ \Gamma_1: G \subset \mathcal{Y}^\alpha_V \rightarrow L^\beta(\Omega_T;V^*) \\
A_1(u)&=u, \hspace{4cm} \Gamma_1(u)=-\tilde{u}-\Lambda u
\end{align*}
\begin{align*}
\nonumber A_2: G \subset \mathcal{Y}^\alpha_V &\rightarrow L^2(\Omega_T;H), \hspace{2cm} \quad \Gamma_2: G \subset \mathcal{Y}^\alpha_V \rightarrow L^2(\Omega_T;H)\\
A_2(u)&= \frac{1}{2}F_u, \hspace{4cm} \Gamma_2(u)= -F_u+ \frac{3}{2}Bu.
\end{align*}
Unfortunately, the coercivity condition  (\ref{superpose-coerc}) required to conclude is not satisfied. We have to therefore perturb the Lagrangian ${\cal L}$ (i.e., essentially perform a stochastic elliptic regularization) as well as the operator $\Gamma_1$ in order to ensure coercivity. We will then let the perturbations go to zero to conclude.
\subsection{Stochastic elliptic regularization}

To do that, we consider the convex lower semi-continuous function on $L^\alpha(\Omega_T, V)$
\begin{align}
\psi(u)=
\begin{cases}
\frac{1}{\beta}\E \int_0^T \Vr \tilde{u}(t) \Vr^\beta_{V^*} dt \quad &\text{if} \  u \in \mathcal{Y}^\alpha_V\\
+\infty & \text{if} \  u \in L^\alpha_{V}(\Omega_T) \backslash \mathcal{Y}^\alpha_V, 
\end{cases}
\end{align}
and  for any $\mu >0$,  its associated self-dual Lagrangian on $L^\alpha_{V}(\Omega_T) \times L^\beta_{V^*}(\Omega_T)$ given by
\begin{equation}
\Psi_\mu(u,p)=\mu\psi(u)+\mu \psi^*(\frac{p}{\mu}).
\end{equation}
We also consider a perturbation operator 
$$Ku:= (\Vr u\Vr^{\alpha-1}_{L^\alpha_V(\Omega_T)}) Du,$$
where $D$ is the duality map between $V$ and $V^*$. Note that by definition, $K$ is a weak-to-weak continuous operator from $\mathcal{Y}^\alpha_V$ to $L^\beta_{V^*}(\Omega_T)$.


\begin{lem}\label{psi-lem} Under the above hypothesis on ${\cal L}$, $B$ and $\Lambda$, 
there exists a process $u_\mu \in \mathcal{Y}^\alpha_V$ such that  $u(0)=u_0, \, \tilde{u}(T)=\tilde{u}(0)=0$, and satisying 
\begin{align*}
\tilde{u}_\mu+Ku_\mu + \Lambda u_\mu +\mu \, \partial \psi(u_\mu) &\in -\bar{\partial}{\cal L}(u_\mu) \\
F_{u_\mu}&=Bu_\mu.
\end{align*} 
\end{lem} 
\begin{proof}
Apply Theorem \ref{superpose-thm}  as follow: Let $Z=\mathcal{Y}^\alpha_V, X_1= L^\alpha(\Omega_T;V), X_2=L^2(\Omega_T;H)$ with $G=\{u \in \mathcal{Y}^\alpha_V; u(0)=u_0\}$ which is a closed linear subspace of $\mathcal{Y}^\alpha_V$, and consider the operators
\begin{align*}
A_1: G \subset \mathcal{Y}^\alpha_V &\rightarrow L^\alpha(\Omega_T;V), \hspace{2cm}   \ \Gamma_1: G \subset \mathcal{Y}^\alpha_V \rightarrow L^\beta(\Omega_T;V^*) \\
A_1(u)&=u, \hspace{4cm} \Gamma_1(u)=-\tilde{u}-\Lambda u- Ku
\end{align*}
\begin{align}\label{2operators}
\nonumber A_2: G \subset \mathcal{Y}^\alpha_V &\rightarrow L^2(\Omega_T;H), \hspace{2cm} \quad \Gamma_2: G \subset \mathcal{Y}^\alpha_V \rightarrow L^2(\Omega_T;H)\\
A_2(u)&= \frac{1}{2}F_u, \hspace{4cm} \Gamma_2(u)= -F_u+ \frac{3}{2}Bu
\end{align}
where their domain is $G$. $A_1, A_2$ are linear, and $\Gamma_1, \Gamma_2$ are weak-to-weak continuous. \\ 
As to the Lagrangians, we  take  on $L^\alpha_{V}(\Omega_T) \times L^\beta_{V^*}(\Omega_T)$, the Lagrangian 
$$
L_1(u,p)=\mathcal{L} \oplus \Psi_\mu (u, p),
$$
while on $L^2_H(\Omega_T)\times L^2_H(\Omega_T)$, we take
\[
L_2(P, Q)=\E \int_0^T M(P(t, w), Q(t, w))\, dt,
\]
where 
$M(P,Q)= \frac{1}{2}\Vr P\Vr^2_H +\frac{1}{2}\Vr Q\Vr^2_H$.\\
In other words, we are considering the functional 
\begin{align*}
I_\mu(u)&=\mathcal{L}\oplus \Psi_\mu (A_1 u,\Gamma_1u)- \E \int_0^T \langle A_1 u, \Gamma_1 u\rangle dt + \E \int_0^T M(A_2 u,\Gamma_2u)- \langle A_2 u, \Gamma_2 u\rangle \, dt\\
&= \mathcal{L}\oplus \Psi_\mu (u,-\tilde{u}-\Lambda u-Ku) - \E \int_0^T \langle u, -\tilde{u}-\Lambda u-Ku \rangle \, dt \\
& \hspace{3.5cm}+ \E \int_0^T M(F_u/2, -F_u+3Bu/2) - \langle F_u/2, -F_u+3Bu/2 \rangle \, dt. 
\end{align*}
We now verify the conditions of Theorem \ref{superpose-thm}.
\begin{align*}
G_0=\text{Ker}(A_2)\cap G = \Big\{u\in  \mathcal{Y}^\alpha_V  ; \  u(t) = u_0+ \int_0^t \tilde{u}(s)  ds, \ \hbox{for some $\tilde{u} \in {L}^\beta_{V^*}(\Omega_T)\Big\}.$}
\end{align*}
It is clear that $A_1(G_0)$ is dense in $L^\alpha(\Omega_T;V)$.
Moreover, $A_2(G)$ is dense in $L^2(\Omega_T;H)$. 
To check the upper semi-continuity of 
 $$
u\to  \E \int_0^T \langle A_1 u, \Gamma_1 u\rangle+ \langle A_2 u, \Gamma_2 u \rangle \, dt
 $$
on  $\mathcal{Y}^\alpha_V$ equipped with the weak topology, we apply It\^o's formula and the fact that  $\langle u,\Lambda u \rangle=0$ on $\mathcal{Y}^\alpha_V$, to obtain 
\begin{align*}
\E \int_0^T \langle A_1 u, \Gamma_1 u\rangle+ \langle A_2 u, \Gamma_2 u \rangle \, dt &= \E \int_0^T \langle u, -\tilde{u}-\Lambda u-Ku \rangle +\langle F_u/2, -F_u+3Bu/2 \rangle \, dt \\
&= \frac{1}{2} \, \E \, \Vr u_0 \Vr_H^2 -  \frac{1}{2} \, \E \, \Vr u(T) \Vr_H^2 -\Vr u\Vr^{\alpha+1}_{L^\alpha_V(\Omega_T)} \\
& \quad +\frac{3}{4} \, \E \int_0^T  \langle F_u(t),Bu(t) \rangle dt.
\end{align*} 
Upper semi-continuity then follows from the compactness of the maps 
$\mathcal{Y}^\alpha_V\to L^2(\Omega; H)$ given by $u\mapsto (u(0), u(T))$, as well as the weak to norm continuity of $B$, which makes the functional $u \mapsto \E \int_0^T  \langle F_u,Bu\rangle dt$ weakly continuous. \\

\noindent To verify the coercivity, we first note that condition (\ref{L10}) implies that for some (different) $C_1>0$, 
$$H_{\mathcal{L}}(0,u) \geq C_1\Big(\Vr u\Vr^\alpha_{L^\alpha_V(\Omega_T)}-1\Big).$$ 
By also taking into account condition (\ref{B-bound}) on $B$, with the fact that $\delta<\frac{\alpha+1}{2}$, we get  that 

\begin{align*}
H_{\mathcal{L}}(0,u)&+ \mu \, \psi(u) +\E \int_0^T \langle u, \tilde{u}+\Lambda u + Ku \rangle \, dt+ \E \int_0^T H_M(0, F_u/2) - \langle F_u/2, -F_u+3Bu/2 \rangle \, dt\\
&= H_{\mathcal{L}}(0,u) + \frac{\mu}{\beta} \Vr \tilde{u}\Vr^\beta_{L^\beta_{V^*}(\Omega_T)} -\frac{1}{2} \, \Vr u_0 \Vr_{L^2(\Omega; H)}^2 + \frac{1}{2} \, \Vr u(T) \Vr_{L^2(\Omega; H)}^2 +\Vr u\Vr^{\alpha+1}_{L^\alpha_V(\Omega_T)} \\
&\qquad + \frac{1}{8} \Vr F_u(t) \Vr^2_{L^2(\Omega_T; H)} -\frac{3}{4} \, \E \int_0^T \langle F_u(t),Bu(t) \rangle \, dt  \\
&\geq C_1\Big(\Vr u\Vr^\alpha_{L^\alpha_V(\Omega_T)}-1\Big) + \frac{\mu}{\beta} \, \Vr \tilde{u}\Vr^\beta_{L^\beta_{V^*}(\Omega_T)} +\Vr u\Vr^{\alpha+1}_{L^\alpha_V(\Omega_T)} \\
& \qquad + C_2 \Big(\Vr F_u(t) \Vr^2_{L^2_H(\Omega_T)} - \Vr F_u \Vr_{L^2_H(\Omega_T)} \Vr Bu \Vr_{L^2_H(\Omega_T)} \Big) +C\\
& \geq  C_1\Big(\Vr u\Vr^\alpha_{L^\alpha_V(\Omega_T)}-1\Big) + \frac{\mu}{\beta} \, \Vr \tilde{u}\Vr^\beta_{L^\beta_{V^*}(\Omega_T)} +\Vr u\Vr^{\alpha+1}_{L^\alpha_V(\Omega_T)}\\
& \qquad  +  C_2\Big(\Vr F_u(t) \Vr^2_{L^2_H(\Omega_T)} - \Vr F_u \Vr_{L^2_H(\Omega_T)} \Vr u \Vr^\delta _{L^\alpha_V(\Omega_T)} \Big) +C \\
& \geq \frac{\mu}{\beta} \, \Vr \tilde{u}\Vr^\beta_{L^\beta_{V^*}(\Omega_T)}+ \Vr u\Vr^{\alpha+1}_{L^\alpha_V(\Omega_T)} \Big(1+ o(\Vr u\Vr_{L^\alpha_V(\Omega_T)})\Big) + C_2 \Vr F_u(t) \Vr^2_{L^2_H(\Omega_T)}.
\end{align*}
Therefore, by Theorem \ref{superpose-thm}, there exists $u_\mu \in G \subset \mathcal{Y}^\alpha_V$ such that $I_\mu(u_\mu)=0$, i.e.
\begin{align*}
0&= \mathcal{L}\oplus \Psi_\mu (u_\mu,-\tilde{u}_\mu-\Lambda u_\mu -Ku_\mu) - \E \int_0^T \langle u_\mu, -\tilde{u}_\mu-\Lambda u_\mu-Ku_\mu \rangle \, dt \\
& \quad + \E \int_0^T M(\frac 1 2 F_{u_\mu}, -F_{u_\mu}+\frac 3 2 Bu_\mu) - \langle \frac 1 2 F_{u_\mu}, -F_{u_\mu}+\frac 3 2 Bu_\mu \rangle \, dt. 
\end{align*}
Since $\mathcal{L}\oplus \Psi_\mu$ is convex and coercive in the second variable, there exists $\bar{r} \in L^\beta_{V^*}(\Omega_T)$ such that
$$\mathcal{L}\oplus \Psi_\mu (u_\mu,-\tilde{u}_\mu -\Lambda u_\mu -Ku_\mu) = \mathcal{L}(u_\mu, \bar{r})+ \Psi_\mu(u_\mu, -\tilde{u}_\mu-\Lambda u_\mu -Ku_\mu-\bar{r}),$$
hence
\begin{align*}
0&= {\cal L}(u_\mu, \bar{r}) - \langle u_\mu, \bar{r} \rangle  + \Psi_\mu(u_\mu, -\tilde{u}_\mu  -\Lambda u_\mu - Ku_\mu-\bar{r}) + \E \int_0^T \langle u_\mu, \tilde{u}_\mu+\Lambda u_\mu +Ku_\mu+\bar{r} \rangle \, dt \\
& \qquad + \E \int_0^T M(\frac 1 2 F_{u_\mu}, -F_{u_\mu}+\frac 3 2 Bu_\mu) - \langle \frac 1 2 F_{u_\mu}, -F_{u_\mu}+\frac 3 2 Bu_\mu \rangle \, dt. 
\end{align*}
Due to the self-duality of ${\cal L}$, $\Psi_\mu$ and $M$, this becomes the sum of three non-negative terms, and therefore 
\begin{align*}
{\cal L}(u_\mu, \bar{r}) &- \E \int_0^T \langle u_\mu(t), \bar{r}(t) \rangle dt =0,\\
\Psi_\mu(u_\mu, -\tilde{u}_\mu -\Lambda u_\mu -Ku_\mu-\bar{r}) &+ \E \int_0^T \langle u_\mu(t), \tilde{u}_\mu(t)+\Lambda u_\mu+Ku_\mu(t)+\bar{r}(t) \rangle \, dt =0,\\
\E \int_0^T M(\frac 1 2 F_{u_\mu}(t), -F_{u_\mu}(t)+&\frac 3 2 Bu_\mu(t)) - \langle \frac 1 2 F_{u_\mu}(t), -F_{u_\mu}(t)+\frac 3 2 Bu_\mu(t) \rangle \, dt=0.
\end{align*}
By the limiting case of Legendre duality, this yields 
\begin{align}\label{mu-lag}
\tilde{u}_\mu+\Lambda u_\mu+Ku_\mu+\mu \, \partial \psi(u_\mu) \in -\bar{\partial}{\cal L}(u_\mu)\\
\nonumber -F_{u_\mu}(t)+\frac 3 2 Bu_\mu(t) \in \bar{\partial} M(t,\frac 1 2 F_{u_\mu}(t))=\frac 1 2 F_{u_\mu}(t).
\end{align}
The second line implies that for a.e. $t\in [0,T]$ we have $\P$-a.s. $F_{u_\mu}= Bu_\mu$. Moreover, from (\ref{mu-lag}) we have that $\partial \psi(u_\mu) \in L^\beta_{V^*}(\Omega_T)$.\\
Now for an arbitrary process $v \in \mathcal{Y}_V^\alpha$ of the form  
$v(t)= v(0)+ \int_0^t \tilde{v}(s) ds +\int_0^t F_v(s) dW(s),$
 we have
$\langle \partial \psi(u_\mu(t)), v \rangle = \langle \Vr \tilde{u}_\mu\Vr^{\beta-2}_{V^*} D^{-1}\tilde{u}, \tilde{v}\rangle$. Applying It\^o's formula with the progressively measurable process $X(t):=\Vr \tilde{u}_\mu\Vr^{\beta-2}_{V^*} D^{-1}\tilde{u}$, we obtain
\begin{align}\label{tilde-integ-byparts}
\nonumber \E \int_0^T \langle\Vr \tilde{u}_\mu\Vr^{\beta-2}_{V^*} D^{-1}\tilde{u}_\mu(t), \tilde{v}(t)  \rangle &= -\E \int_0^T \langle \frac{d}{dt}(\Vr \tilde{u}_\mu\Vr^{\beta-2}_{V^*} D^{-1}\tilde{u}_\mu) , v(t)\rangle \\
\nonumber &\quad +\E \, \langle \Vr \tilde{u}_\mu(T)\Vr^{\beta-2}_{V^*} D^{-1}\tilde{u}_\mu(T),v(T) \rangle \\
&\quad -\E \, \langle \Vr \tilde{u}_\mu(0)\Vr^{\beta-2}_{V^*} D^{-1}\tilde{u}_\mu(0),v(0) \rangle,
\end{align}
which, in view of (\ref{mu-lag}), implies that
\begin{align*}
0&=\E \int_0^T \Big[\langle \tilde{u}_\mu(t)+\Lambda u_\mu+Ku_\mu(t)   + \bar{\partial}\mathcal{L}(u_\mu), v \rangle+\mu \langle\Vr \tilde{u}_\mu\Vr^{\beta-2}_{V^*} D^{-1}\tilde{u}_\mu, \tilde{v}  \rangle \Big] \, dt\\
&= \E \int_0^T \Big\langle \tilde{u}_\mu(t)+\Lambda u_\mu+Ku_\mu(t)   +  \bar{\partial}\mathcal{L}(u_\mu)-\mu \frac{d}{dt}(\Vr \tilde{u}_\mu\Vr^{\beta-2}_{V^*} D^{-1}\tilde{u}_\mu)  , v \Big\rangle \, dt \\
& \quad +\mu \, \E \, \langle \Vr \tilde{u}_\mu(T)\Vr^{\beta-2}_{V^*} D^{-1}\tilde{u}_\mu(T),v(T) \rangle -\mu \, \E \, \langle \Vr \tilde{u}_\mu(0)\Vr^{\beta-2}_{V^*} D^{-1}\tilde{u}_\mu(0),v(0) \rangle,
\end{align*}
hence $\tilde{u}_\mu(T)= \tilde{u}_\mu(0)=0$ and 
$\tilde{u}_\mu+\Lambda u_\mu+Ku_\mu -\mu \frac{d}{dt}(\Vr \tilde{u}_\mu\Vr^{\beta-2}_{V^*} D^{-1}\tilde{u}_\mu) \in -\bar{\partial}{\cal L}(u_\mu).
$
\end{proof}
\noindent In the following lemma, we shall remove the regularizing term $\mu \partial \psi $.
\begin{lem}\label{mu-lemma} Under the above assumptions on ${\cal L}$, $\Lambda$ and $B$,   
there exists $u\in \mathcal{Y}^\alpha_V$ such that $u(0)=u_0$, $F_u=Bu$ and 
\begin{align}
\mathcal{L}( u,-\tilde{u}-\Lambda u -Ku) + \E \int_0^T \langle  u(t), \tilde{u}(t)+\Lambda u+Ku(t) \rangle \, dt&=0.
\end{align}
\end{lem}
\begin{proof}
Lemma \ref{psi-lem} yields that for every $\mu>0$ there exist $u_\mu \in \mathcal{Y}^\alpha_V$ such that $u_\mu (0)=u_0$, $\tilde{u}_\mu(T)=\tilde{u}_\mu(0)=0$, and satisfying
\begin{align}\label{multiply-eqt}
\tilde{u}_\mu+\Lambda u_\mu+Ku_\mu +\mu \, \partial \psi(u_\mu) &\in -\bar{\partial}\mathcal{L}(u_\mu) \\
\nonumber F_{u_\mu}(t)=Bu_\mu(t).
\end{align}  
Now we show that $u_\mu$ is bounded in $\mathcal{Y}^\alpha_V$ with bounds independent of $\mu$. Indeed, multiplying (\ref{multiply-eqt}) by  $u_\mu$ and integrating over $\Omega\times [0,T]$, we obtain
\begin{align*}
\E \int_0^T \Big\langle \tilde{u}_\mu(t)+\Lambda u_\mu+Ku_\mu(t) +\mu \, \partial \psi(u_\mu(t)),u_\mu \Big\rangle &=- \E \int_0^T \langle \bar{\partial}\mathcal{L}(u_\mu), u_\mu\rangle dt.  
\end{align*}
Apply It\^o's formula and use the fact that $\E \int_0^T \langle \mu \, \partial \psi(u_\mu(t)), u_\mu \rangle \, dt \geq 0$ and $\E \int_0^T \langle \Lambda u(t),u(t) \rangle dt = 0$ to get
\begin{align*}
 -\frac{1}{2} \, \Vr u_{\mu,0} \Vr_{L^2(\Omega; H)}^2 + \frac{1}{2} \, \Vr u_\mu(T) \Vr_{L^2(\Omega; H)}^2 &-\frac{1}{2}  \Vr F_{u_{\mu}} \Vr^2_{L^2_H(\Omega_T)}+\Vr u_\mu\Vr^{\alpha+1}_{L^\alpha_V(\Omega_T)} \\ &= -\E \int_0^T \langle \mu \, \partial \psi(u_\mu)+\bar{\partial}\mathcal{L}(u_\mu ), u_\mu \rangle \, dt\\
 & \leq -\E \int_0^T \langle \bar{\partial}\mathcal{L}(u_\mu), u_\mu \rangle \, dt.
\end{align*}
Since for $u_\mu \in \mathcal{Y}^\alpha_V$ we have $u_\mu \in L^\infty(0,T;H)$, then in view of (\ref{L20}), we get 
\begin{align*}
C_1 + \Vr u_\mu\Vr^{\alpha+1}_{L^\alpha_V(\Omega_T)} &\leq \Vr \bar{\partial}\mathcal{L}(u_\mu) \Vr_{L^\beta_{V^*}(\Omega_T)} \Vr u_\mu \Vr_{L^\alpha_V(\Omega_T)}\\
&\leq C \, \Vr u_\mu \Vr^\alpha _{L^\alpha_V(\Omega_T)}.
\end{align*}
The above inequality implies that $\Vr u_\mu \Vr_{L^\alpha_V(\Omega_T)}$ is bounded.\\
Next, we multiply (\ref{multiply-eqt}) by $D^{-1}\tilde{u}_\mu$ and integrate over $\Omega_T$ to get that
\begin{align*}
0&=\E \int_0^T \Big\langle \tilde{u}_\mu(t)+ \Lambda u_\mu+Ku_\mu(t) +\mu \, \partial \psi(u_\mu(t)) +\bar{\partial}\mathcal{L}(t,u_\mu),D^{-1}\tilde{u}_\mu \Big\rangle \, dt
\end{align*}
From (\ref{tilde-integ-byparts}), and choosing $v=\Vr\tilde{u}_\mu\Vr^{\beta-2}_{V^*} D^{-1}\tilde{u}_\mu$ with $\tilde{v}= \frac{d}{dt}(\Vr\tilde{u}_\mu\Vr^{\beta-2}_{V^*} D^{-1}\tilde{u}_\mu)$ and $F_v=0$, we get that $\,\E \int_0^T \langle \partial \psi(u_\mu(t)), D^{-1}\tilde{u}_\mu \rangle \, dt =0$, which together with condition(\ref{L20}) imply that
\begin{align*}
\Vr \tilde{u}_\mu\Vr^2_{L^\beta_{V^*}} \leq \Vr \Lambda u_\mu \Vr_{L^\beta_{V^*}} \Vr \tilde{u}_\mu \Vr_{L^\beta_{V^*}} +\Vr Ku_\mu \Vr_{L^\beta_{V^*}} \Vr \tilde{u}_\mu \Vr_{L^\beta_{V^*}} + C \, \Vr u_\mu \Vr^{\alpha-1}_{L^\alpha_V} \Vr \tilde{u}_\mu \Vr_{L^\beta_{V^*}},
\end{align*}
hence by condition (\ref{Lambda-cond}) we have
\begin{align*}
\Vr \tilde{u}_\mu\Vr_{L^\beta_{V^*}(\Omega_T)} \leq f(\Vr u_\mu \Vr_{L^\alpha_V(\Omega_T)})+\Vr Ku_\mu \Vr_{L^\beta_{V^*}(\Omega_T)} + C\, \Vr u_\mu \Vr^{\alpha-1}_{L^\alpha_V(\Omega_T)}
\end{align*}
which means that $\Vr \tilde{u}_\mu\Vr_{L^\beta_{V^*}(\Omega_T)}$ is bounded. From (\ref{B-bound}) and  since $F_{u_\mu}=Bu_\mu$ we deduce that $\Vr F_{u_\mu}\Vr_{L^2_{H}(\Omega_T)}$ is also bounded.
Now since $(u_\mu)_\mu$ is bounded in $\mathcal{Y}^\alpha_V$, there exists $u\in \mathcal{Y}^\alpha_V$ such that $u_\mu \rightharpoonup u$ weakly in $\mathcal{Y}^\alpha_V$, which means that $u_\mu \rightharpoonup u$ weakly in $L^\alpha_V(\Omega_T)$, $\tilde{u}_\mu \rightharpoonup \tilde{u}$ weakly in $L^\beta_{V^*}(\Omega_T)$, and $F_{u_\mu} \rightharpoonup F_u$ weakly in $L^2_H(\Omega_T)$. From (\ref{multiply-eqt}) and since $B$ is weak-to-norm continuous we have $F_u=Bu$. Then, by (\ref{mu-lag}) we obtain
\begin{align*}
0&= \mathcal{L}( u_\mu  ,-\tilde{u}_\mu-\Lambda u_\mu-Ku_\mu -\mu \, \partial \psi(u_\mu)) \\
& \qquad + \E \int_0^T \Big\langle  u_\mu(t), \tilde{u}_\mu(t)+\Lambda u_\mu+Ku_\mu(t) +\mu \, \partial \psi(u_\mu(t)
\Big\rangle \, dt\\
& \geq \mathcal{L}( u_\mu,-\tilde{u}_\mu-\Lambda u_\mu -Ku_\mu -\mu \, \partial \psi(u_\mu)) + \E \int_0^T \langle  u_\mu(t), \tilde{u}_\mu(t)+\Lambda u_\mu+Ku_\mu(t) \rangle \, dt.
\end{align*}
Since $\Lambda+K$ is weak-to-weak continuous, $\langle \partial \psi(u_\mu)), u_\mu \rangle = \Vr \tilde{u}_\mu \Vr^\beta_{L^\beta_{V^*}}$ is uniformly bounded, and ${\cal L}$ is weakly lower semi-continuous on $L^\alpha_V\times L^\beta_{V^*}$,  we get
\begin{align*}
0 &\geq \liminf_{\mu \rightarrow 0}\, \mathcal{L}( u_\mu,-\tilde{u}_\mu-\Lambda u_\mu-Ku_\mu -\mu \, \partial \psi(u_\mu))+ \E \int_0^T \langle  u_\mu(t), \tilde{u}_\mu(t)+\Lambda u_\mu+Ku_\mu(t) \rangle \, dt\\
&\geq \mathcal{L}( u,-\tilde{u}-\Lambda u-Ku) + \E \int_0^T \langle  u(t), \tilde{u}(t)+\Lambda u+Ku(t) \rangle \, dt.
\end{align*}
Since ${\cal L}$ is a self-dual Lagrangian on $L^\alpha_V\times L^\beta_{V^*}$, the reverse inequality is always true, and therefore
$$\mathcal{L}( u,-\tilde{u}-\Lambda u-Ku) + \E \int_0^T \langle u(t), \tilde{u}(t)+\Lambda u+Ku(t) \rangle \, dt=0,$$
which completes the proof of Lemma \ref{mu-lemma}.
\end{proof}
\subsection{End of proof of Theorem \ref{last-thm}}
We shall work toward eliminating the perturbation $K$. By Lemma \ref{mu-lemma}, for each $\varepsilon >0$, there exists a $u_\varepsilon \in G$ such that $F_{u_\varepsilon}=Bu_\varepsilon$ and
\begin{equation}\label{epsilon-eqt1}
\mathcal{L}( u_\varepsilon,-\tilde{u}_\varepsilon-\Lambda u_\varepsilon-\varepsilon Ku_\varepsilon) + \E \int_0^T \langle  u_\varepsilon(t), \tilde{u}_\varepsilon(t)+\Lambda u_\varepsilon+\varepsilon Ku_\varepsilon(t) \rangle \, dt=0,
\end{equation}
or equivalently
\begin{equation}\label{epsilon-eqt2}
\tilde{u}_\varepsilon+\Lambda u_\varepsilon+\varepsilon Ku_\varepsilon \in -\bar{\partial} \mathcal{L}(u_\varepsilon).
\end{equation}
Similar to the argument in Lemma \ref{mu-lemma} we show that $u_\varepsilon$ is bounded in $\mathcal{Y}^\alpha_V$ with bounds independent of $\varepsilon$. First, we multiply (\ref{epsilon-eqt2}) by $u_\varepsilon$ and integrate over $\Omega_T$ to obtain
\begin{align*}
\E \int_0^T \langle \tilde{u}_\varepsilon(t)+\Lambda u_\varepsilon+\varepsilon Ku_\varepsilon(t), u_\varepsilon(t)  \rangle \, dt &= -\E \int_0^T \langle \bar{\partial} \mathcal{L}(u_\varepsilon), u_\varepsilon \rangle \, dt \\
& \leq \Vr \bar{\partial}{{\cal L}}(u_\varepsilon) \Vr_{L^\beta_{V^*}(\Omega_T)} \Vr u_\varepsilon \Vr_{L^\alpha_V(\Omega_T)}\\
&\leq C \, \Vr u_\varepsilon \Vr^\alpha _{L^\alpha_V(\Omega_T)},
\end{align*}
where we used (\ref{L20}). In view of (\ref{epsilon-eqt1}) and (\ref{L10}), this implies that
$$C (\Vr u_\varepsilon \Vr^\alpha_{L^\alpha_V(\Omega_T)}-1) \leq \mathcal{L}( u_\varepsilon,-\tilde{u}_\varepsilon-\Lambda u_\varepsilon-\varepsilon Ku_\varepsilon)\leq C \, \Vr u_\varepsilon \Vr^\alpha_{L^\alpha_V(\Omega_T)},$$
from which we deduce that $u_\varepsilon$ is bounded in $L^\alpha_V(\Omega_T)$. Next, we multiply (\ref{epsilon-eqt2}) by $D^{-1}\tilde{u}_\varepsilon$ to obtain
\begin{align*}
\E \int_0^T \langle \tilde{u}_\varepsilon(t)+\Lambda u_\varepsilon+\varepsilon Ku_\varepsilon(t), D^{-1}\tilde{u}_\varepsilon(t)  \rangle &= -\E \int_0^T \langle \bar{\partial} \mathcal{L}(u_\varepsilon), D^{-1}\tilde{u}_\varepsilon(t) \rangle \, dt, 
\end{align*}
and therefore similar to the reasoning as in Lemma \ref{mu-lemma} we deduce that
\begin{align*}
\Vr \tilde{u}_\varepsilon\Vr_{L^\beta_{V^*}(\Omega_T)} \leq f\Big(\Vr u_\varepsilon \Vr_{L^\alpha_V(\Omega_T)}\Big)+\varepsilon \Vr Ku_\varepsilon \Vr_{L^\beta_{V^*}(\Omega_T)} + C \, \Vr u_\varepsilon \Vr^{\alpha-1}_{L^\alpha_V(\Omega_T)}.
\end{align*}
Hence $\tilde{u}_\varepsilon$ is bounded in $L^\beta_{V^*}(\Omega_T)$, and there exists $u \in \mathcal{Y}_V^\alpha$ such that $u_\varepsilon \rightharpoonup u$ weakly in $L^\alpha_V(\Omega_T)$, and $\tilde{u}_\varepsilon \rightharpoonup \tilde{u}$ weakly in $L^\beta_{V^*}(\Omega_T)$, and $F_{u_\varepsilon} \rightharpoonup F_u$ weakly in $L^2_H(\Omega_T)$. Moreover,  
\begin{align*}
0&= \mathcal{L}( u_\varepsilon,-\tilde{u}_\varepsilon-\Lambda u_\varepsilon-\varepsilon Ku_\varepsilon)) + \E \int_0^T \langle  u_\varepsilon(t), \tilde{u}_\varepsilon(t)+\Lambda u_\varepsilon+Ku_\varepsilon(t) \rangle  \, dt\\
& \geq \mathcal{L}( u_\varepsilon,-\tilde{u}_\varepsilon-\Lambda u_\varepsilon-\varepsilon Ku_\varepsilon) + \E \int_0^T \langle  u_\varepsilon(t), \tilde{u}_\varepsilon(t)+\Lambda u_\varepsilon\rangle \, dt.
\end{align*}
Again, ${\cal L}$ is weakly lower semi-continuous on $L^\alpha_V\times L^\beta_{V^*}$, therefore by letting $\varepsilon \rightarrow 0$ we get
\begin{align*}
0 \geq \mathcal{L}( u,-\tilde{u}-\Lambda u) + \E \int_0^T \langle  u(t), \tilde{u}(t)+\Lambda u \rangle \, dt.
\end{align*}
Since the reverse inequality is always true we have
$$\mathcal{L}( u,-\tilde{u}-\Lambda u)+ \E \int_0^T \langle u(t), \tilde{u}(t)+\Lambda u \rangle \, dt=0,$$
and also $F_u(t)=Bu(t)$. By the limiting case of Legendre duality, we now have for a.e. $t\in [0,T]$, $\P$-a.s.
$\tilde{u}+\Lambda u  \in -\bar{\partial}\mathcal{L}( u)$, integrating over $[0,t]$ with the fact that $\int_0^t \tilde{u}(s) ds =u(t)- u_0 -\int_0^t F_u(s) dW(s)$, and $F_u(t)=Bu(t)$ we obtain 
\begin{align*}
u(t)&=u_0-\int_0^t \Big(\bar{\partial}{{\cal L}}(u)(s)+\Lambda u(s)\Big) ds + \int_0^t B(u(s))dW(s).
\end{align*}
\section{Non-additive noise driven by monotone vector fields\label{last.section}}


Consider the following type of equations
 \begin{equation}\label{last-SDE}
\begin{cases}
du(t)= -A (t,u(t)) dt -\Lambda u(t) dt+ B(t, u(t)) dW(t)\\
u(0)=u_0,
\end{cases}
\end{equation}
where $V\subset H\subset V^*$ is a Gelfand triple, and 
$A:\Omega\times [0,T] \times V \rightarrow V^*$ and $B:\Omega\times [0,T] \times V \rightarrow H$
are progressively measurable. 

\begin{thm} \label{mono-thm} Assume $A: D(A)\subset V\to V^*$ is a progressively measurable $\Omega_T$-dependent maximal monotone operator satisfying condition (\ref{A-cond2}) with $\alpha>1$ and its conjugate $\beta$, as well as 
\begin{equation}\label{A-Cond1}
\Vr A_{w,t}u\Vr_{V^*} \leq k (\omega,t)(1+\Vr u \Vr_V)\quad \hbox{for all $u\in V$, $dt \otimes \P$ a.s.}
\end{equation}
for some $k\in L^\infty(\Omega_T)$.\\
Let $B:  \mathcal{Y}^\alpha_V \rightarrow L^2(\Omega_T;H)$ be a weak-to-norm continuous map satisfying (\ref{B-bound}) and  $\Lambda: \mathcal{Y}^\alpha_V\to L^\beta(\Omega_T;V^*)$  is a weak-to-weak continuous map satisfying (\ref{Lambda-cond.0}) and (\ref{Lambda-cond}). \\
Let $u_0$ be a given random variable in $L^2_H(\Omega,\F_0,\mathbb{P};H)$, then equation (\ref{last-SDE}) 
has a variational solution in $ \mathcal{Y}^\alpha_V$.
\end{thm}
\begin{proof} Associate again to  $A_{\omega,t}$ an $\Omega_T$-dependent self-dual Lagrangian $L_{A_{\omega,t}}(u,p)$ on $V \times V^*$ in such a way that  for almost every $t\in [0, T]$, $\P$-a.s. 
we have $A_{\omega,t}=\bar{\partial}L_{A_{\omega,t}}$. 
Then by Lemma \ref{lem-AtoL},
the Lagrangian 
$$\mathcal{L}_A(u,p)=\E \int_0^T L_{A_{\omega,t}}(u(\omega,t),p(\omega,t))dt$$
is self-dual on $L^\alpha(\Omega_T;V) \times L^\beta(\Omega_T;V^*)$, 
and satisfies \begin{align*}\label{L1}
 C_1 (\Vr u \Vr^\alpha_{L^\alpha_V(\Omega_T)}+\Vr p \Vr^\beta_{L^\beta_{V^*}(\Omega_T)}-1) \leq  \mathcal{L}(u,p)
\leq C_2(1+\Vr u \Vr^\alpha_{L^\alpha_V(\Omega_T)}+\Vr p \Vr^\beta_{L^\beta_{V^*}(\Omega_T)}).
\end{align*}
(\ref{A-Cond1}) also implies that for some $C_3>0$,
\begin{equation*}\label{2}
\Vr \bar{\partial}{\cal L}_{\cal A}(u)\Vr_{L^\beta_{V^*}(\Omega_T)} \leq C_3(1+ \Vr u \Vr^{\alpha-1}_{L^\alpha_V(\Omega_T)}).
\end{equation*}
The rest follows from Theorem \ref{last-thm}.
\end{proof}


The first immediate application is the following case when the equation is driven by the gradient of a convex function. 
\begin{thm}\label{phi-not-add} Let $V\subset H\subset V^*$ be a Gelfand triple, and let $\varphi: V\to \R\cup\{+\infty\}$ be an $\Omega_T$-dependent  convex lower semi-continuous function on $V$ such that for $\alpha >1$ and some constants $C_1, C_2>0$, for every $t\in [0, T]$, $\mathbb{P}$-a.s. we have
\begin{equation}\label{phi-bound}
C_2(\|u\|^\alpha_{L^\alpha_V(\Omega_T)}-1) \leq \E \int_0^T \varphi (t, u(t)) \, dt \leq C_1(1+\|u\|^\alpha_{L^\alpha_V(\Omega_T)}).
\end{equation}
and 
\begin{equation}\label{pphi}
\Vr {\partial}\varphi(u)\Vr_{L^\beta_{V^*}(\Omega_T)} \leq C_3(1+ \Vr u \Vr^{\alpha-1}_{L^\alpha_V(\Omega_T)}).
\end{equation}
Consider the equation
\begin{equation}\label{simple}
\begin{cases}
du(t)+\Lambda u(t) dt=-\partial \varphi (u(t) dt+ B(u(t))\, dW(t)\\
u(0)=u_0, 
\end{cases}
\end{equation}
where $B:  \mathcal{Y}^\alpha_V \rightarrow L^2(\Omega_T;H)$ be a weak-to-norm continuous map satisfying (\ref{B-bound}) and  $\Lambda: \mathcal{Y}^\alpha_V\to L^\beta(\Omega_T;V^*)$  is a weak-to-weak continuous map satisfying (\ref{Lambda-cond.0}) and (\ref{Lambda-cond}). 
Let $u_0$ be a random variable in $L^2(\Omega,\F_0,\mathbb{P};H)$, then Equation (\ref{simple}) 
has a solution $u$ in $ \mathcal{Y}^\alpha_V$.
\end{thm}
\begin{ex}
Let $D \subset \R^n$ be an open bounded domain, then the  SPDE
\begin{equation}\label{exm-delta}
\begin{cases}
du(t)=\Delta u \, dt+ |u|^{q-1}u \, dW  \qquad & \textnormal{in} \ [0,T] \times D\\
u(t, x)=0  \qquad & \textnormal{in} \ [0,T] \times \partial D\\
u(0, x)=u_0(x)  \qquad & \textnormal{on} \ D,
\end{cases}
\end{equation}
has a solution provided $\frac{1}{2}\leq q < \frac{n}{n-2}$.
\begin{proof}
We take $Bu = |u|^{q-1}u$ and set $H=L^2(D)$.\\

\noindent 1) $\frac{1}{2}\leq q \leq 1$.  Take $\alpha=2$, $V=H^1_0(D), H=L^2(D)$ and $\varphi (u)=\frac{1}{2}\int_D|\nabla u|^2$. As long as $1\leq2q <2^*$, that is $\frac{1}{2}\leq q<\frac{n}{n-2}$, then $B$ is weak-to-norm continuous from $\mathcal{Y}^2_V$ to $L^2(\Omega_T,H)$. Since $0\leq 4q-2\leq 2$ we have 
\begin{align*}
\Vr Bu \Vr_{L^2_H(\Omega_T)} =\left(\E \int_0^T \Vr u^q \Vr^2_{L^2(D)} \, dt\right)^{\frac{1}{2}}
  \leq C \Vr u \Vr_{L^2_{V}}^{\frac{1}{2}} \ \Vr u \Vr_{L^{4q-2}_{V}}^{q-\frac{1}{2}}
 \leq C \Vr u \Vr_{L^2_{V}}^{\frac{1}{2}} \ \Vr u \Vr_{L^{2}_{V}}^{q-\frac{1}{2}}
 \leq C \Vr u \Vr_{L^2_{V}}^q, 
\end{align*}
which is the condition required by Theorem \ref{phi-not-add} with $\delta =q <\frac{3}{2}=\frac{\alpha+1}{2}$. Hence, Equation (\ref{exm-delta}) has a solution $u \in \mathcal{Y}^2_V$.\\

\noindent  2) $1<q< \frac{n}{n-2}$. In this case, take $\alpha=4q-2$ then since $2< 4q-2$, (\ref{B-bound-exm}) reduces to
\[
\Vr Bu \Vr_{L^2_H(\Omega_T)} \leq C \Vr u \Vr_{L^{4q-2}_{V}}^{q},
\]  
Note that $\delta=q< 2q-\frac{1}{2}=\frac{\alpha +1}{2}$ which verifies the condition on $B$. However, the Lagrangian (here the convex function $\varphi$) is not coercive on the space ${\cal Y}_V^\alpha={\cal Y}_V^{4q-2}$. To remedy this, we add a perturbation that makes the Lagrangian coercive on this space by considering the convex function
\[
\varphi_\epsilon (u)=\frac{1}{2}\int_D|\nabla u|^2\, dx + \frac{\epsilon}{4q-2}\int_D|\nabla u|^{4q-2}\, dx.
\]
Now we apply Theorem \ref{phi-not-add} with $\alpha=4q-2$, $V=W^{1,\alpha}_0(D), H=L^2(D)$, and $\varphi_\epsilon$ to get a solution $u_\epsilon$ for the equation
\begin{equation} 
\begin{cases}
du(t)=(\Delta u +\epsilon \Delta_{4q-2}u) \, dt+ |u|^{q-1}u \, dW  \qquad & \textnormal{in} \ [0,T] \times D\\
u(t, x)=0  \qquad & \textnormal{in} \ [0,T] \times \partial D\\
u(0, x)=u_0(x)  \qquad & \textnormal{on} \ D.
\end{cases}
\end{equation}
With a similar argument as what we have already done (twice) in the proof of Theorem (\ref{last-thm}), we let $\epsilon$ go to zero and get a solution for (\ref{exm-delta}) in $ \mathcal{Y}^2_V$. 
\end{proof}
\end{ex}
\begin{ex} \textbf{Stochastic Navier Stokes equation in dimension 2}\\
Consider the incompressible Navier-Stokes equation on a bounded smooth domain $D\subset\R^2$
\begin{equation}\label{exm-navier}
\begin{cases}
du(t)+ (u\cdot \nabla)u\, dt=\nu \Delta u \, dt+\nabla p\, dt+  B(u(t)) \, dW(t) \qquad &\text{on} \, [0,T]\times D\\
\div\, u=0 \qquad &\text{on} \, [0,T]\times D\\ 
u(t, x)=0 \qquad &\text{on} \, [0,T]\times \partial D\\
u(0, x)=u_0(x) \qquad &\text{on} \,  D,
\end{cases}
\end{equation}
where $\nu>0$. Setting $V=\{u \in H^1_0(D;\R^2); \, \div \, u=0 \}$, $H=L^2(D)$ and $\alpha=2$, we assume $B:  \mathcal{Y}^2_V \rightarrow L^2(\Omega_T;H)$ is a weak-to-norm continuous map satisfying (\ref{B-bound}).\\
The functional $\varphi$ and the nonlinear operator $\Lambda$ are defined on $V$ by
\begin{equation}
\varphi(u)= \frac{\nu}{2} \int_D \sum_{j,k=1}^2 \left(\frac{\partial u_j}{\partial x_k}\right)^2 dx \qquad \text{and} \qquad \Lambda u:=(u\cdot \nabla) u.
\end{equation}
The functional $\varphi$ is convex on $V$ and satisfies (\ref{phi-bound}) and (\ref{pphi}). Note that 
$\partial \varphi (u)=-\nu\Delta u +\nabla p$. 
It is also standard to show that $\langle \Lambda u, u\rangle=0$ for $u\in V$. Lifting $\Lambda$ to the path space, one can show that $\Lambda:  \mathcal{Y}^2_V \to L^2_{V^*}$ is weak-to-weak continuous. In fact, let $u^n \rightharpoonup u$ in $\mathcal{Y}^2_V$ then for a fixed $v\in \mathcal{Y}^2_V$   we have that
$$\E \int_0^T \langle \Lambda u^n, v\rangle=\E \int_0^T \int_D \sum_{j,k=1}^2 u_k^n\frac{\partial u^n_j}{\partial x_k} v_j \, dx dt = -\E \int_0^T \int_D \sum_{j,k=1}^2 u_k^n\frac{\partial v_j}{\partial x_k} u^n_j \, dx dt.$$
Moreover, we have the following standard estimate (see for example \cite{G-book}):
$$\Vr \Lambda u^n \Vr_{V^*} \leq C \Vr u^n\Vr_H \Vr u^n\Vr_V,$$
which, due to the fact that $\mathcal{Y}^2_V \subset C(0,T;H)$ continuously, translates to
$$\Vr \Lambda u^n \Vr_{L^2_{V^*}} \leq C \Vr u^n\Vr_{C(0,T;H)} \Vr u^n\Vr_{L^2_V}.$$
Therefore, $\Lambda u^n \rightharpoonup \Lambda u$ in $L^2_{V^*}$ and condition (\ref{Lambda-cond}) holds. 
Applying Theorem (\ref{last-thm}), we deduce that firstly $F_u=Bu$ and also the infimum of the functional 
$$I(u)=\E \int_0^T \Big( \varphi(t,u(t)) + \varphi^*(t,-\tilde{u}(t)- (u\cdot \nabla)u(t)) + \langle u(t), \tilde{u}(t)+(u\cdot \nabla)u(t)\rangle \Big) dt,$$
on $\mathcal{Y}^2_V$ is zero and is attained at a solution $u$ of (\ref{exm-navier}).
\end{ex}
\subsection*{Non-additive noise driven by monotone vector fields in divergence form}\label{div-form}
We now show the existence of a variational solution to the following equation: 
\begin{equation}\label{div-SDE.1}
\begin{cases}
{du}= \div(\beta(\nabla u(t,x)))dt +B(u(t))dW(t)  \qquad & \textnormal{in} \ [0,T] \times D\\
u(0,x)=u_0 &\on \ \partial D, 
\end{cases}
\end{equation}
where 
$D$ is a bounded domain in $\R^n$ and the initial position  $u_0$ belongs to $L^2(\Omega, \F_0,\P; L^2(D))$. We assume that 
\begin{enumerate}
\item The $\Omega_T$-dependent vector field $\beta: \R^n \to \R^n$  is progressively measurable and maximal monotone  such that 
for functions $c_1,c_2,c_3 \in L^\infty(\Omega_T)$, and $m_1, m_2\in L^1(\Omega_T)$, it satisfies $dt \otimes \P$-a.s.    
\begin{align}\label{A-cond20}
 \langle \beta(x),x\rangle \geq \max \{ c_1\Vr x \Vr_{\R^n}^2 -m_1, \, c_2\Vr \beta (x) \Vr^2_{\R^n}-m_2\}\quad \hbox{for all $x\in \R^n$},
\end{align}
and
\begin{equation}\label{A-cond10}
\Vr \beta (x)\Vr_{\R^n} \leq c_3(1+\Vr x \Vr_{\R^n})\quad \hbox{for  all $x\in \R^n$, }
\end{equation}
\item The operator $B:  \mathcal{Y}^2_{H^1_0(D)} \rightarrow L^2(\Omega_T;L^2(D))$ is a  weak-to-norm continuous map such that for some 
 $C>0$ and $0<\delta< \frac{3}{2}$, 
\begin{equation*}\label{B-bound-div}
\Vr Bu \Vr_{L^2_{L^2(D)}(\Omega_T)} \leq C \Vr u\Vr^\delta_{L^2_{H_0^1(D)}(\Omega_T)} \quad\hbox{for any $u\in \mathcal{Y}^2_{H^1_0(D)}$.}
\end{equation*}
\end{enumerate}
\begin{thm} \label{last} Under the above conditions on $\beta$ and $B$, Equation (\ref{div-SDE.1}) has a variational solution.
\end{thm}
We shall need the following lemma, which associates to an $\Omega_T$-dependent self-dual Lagrangian on $\R^n\times \R^n$, a self-dual Lagrangian on 
$L^2(\Omega_T; H^1_0(D)) \times  L^2(\Omega_T; H^{-1}(D))$.

\begin{lem}\label{sd-div-lag} Let  $L$ be an $\Omega_T$-dependent self-dual Lagrangian on $\R^n\times \R^n$, then the Lagrangian defined  by
\begin{equation*}\label{div-lag}
\mathscr{L}(u,p)= \inf \left\{ \E \int_0^T \int_D L (\nabla u(t,x), f(t,x))\, dx \, dt; \ f \in  L^2(\Omega_T;L^2_{\R^n}(D)), -\ntx{div} (f)=p  \right\}
\end{equation*}
is self-dual on $L^2(\Omega_T; H^1_0(D)) \times  L^2(\Omega_T; H^{-1}(D))$.
\end{lem}
We shall need the following general lemma.

\begin{lem}\label{abstract lemma}
Let $L$ be a self-dual Lagrangian on a Hilbert space $\mathcal{H} \times \mathcal{H}$, and let $\Pi:\mathcal{V} \rightarrow \mathcal{H}$ be a bounded linear operator from a reflexive Banach space $\mathcal{V}$ into $\mathcal{H}$ such that the operator $\Pi^*\Pi$ is an isomorphism from $\mathcal{V}$ into $\mathcal{V}^*$. Then, the Lagrangian 
\begin{equation*}
\mathcal{L}(u,p)= \inf \left\{ L( \Pi u, f); \ f \in \mathcal{H} , \Pi^* (f)=p  \right\},
\end{equation*}
is self-dual on $\mathcal{V} \times \mathcal{V}^*$.
\end{lem}
\begin{proof}
For a fixed $(q,v) \in \mathcal{V}^* \times \mathcal{V}$, write
\begin{align*}
\mathcal{L}^*(q,v)&=\sup \Big\{ \langle q,u \rangle + \langle v,p\rangle- \mathcal{L}(u,p); \ u \in \mathcal{V}, \, p\in \mathcal{V}^* \Big\}\\
&=\sup \Big\{ \langle q,u \rangle + \langle v,p\rangle- L(\Pi u,f); \ u \in \mathcal{V}, \, p\in \mathcal{V}^*,   f \in \mathcal{H} , \Pi^* (f)=p \Big\}\\
&=\sup \Big\{ \langle q,u \rangle + \langle v,-\Pi^*f \rangle- L(\Pi u,f); \ u \in \mathcal{V},   f \in \mathcal{H} \Big\}\\
&=\sup \Big\{ \langle q,u \rangle + \langle \Pi v,f \rangle- L(\Pi u,f); \ u \in \mathcal{V},   f \in \mathcal{H} \Big\}.
\end{align*}
Since $\Pi^*\Pi$ is an isomorphism, for $q \in \mathcal{V}^*$ there exists a fixed $f_0 \in \mathcal{H}$ such that $\Pi^*f_0=q$.
Moreover, the space $$\mathcal{E}=\{g \in \mathcal{H}; \ g=\Pi u, \ \textnormal{for some} \ u\in  \mathcal{V}\},$$ is closed in ${\cal H}$ in such a way that its indicator function $\chi_{\mathcal{E}}$ on $\mathcal{H}$
$$\chi_{\mathcal{E}}(g)= \begin{cases}
0 \qquad & g\in \mathcal{E}\\
+\infty &\textnormal{elsewhere},
\end{cases}
$$
is convex and lower semi-continuous. Its Legendre transform is then given for each $f\in \mathcal{H}$ by
$$\chi^*_{\mathcal{E}}(f)= \begin{cases}
0 \qquad &\Pi^*f=0\\
+\infty &\textnormal{elsewhere}.
\end{cases}$$
It follows that 
\begin{align*}
\mathcal{L}^*(q,v)&=\sup \Big\{ \langle f_0,\Pi u \rangle + \langle \Pi v,f \rangle- L(\Pi u,f); \ u \in \mathcal{V},   f \in \mathcal{H} \Big\}\\
&=\sup \Big\{ \langle f_0,g \rangle + \langle \Pi v,f \rangle- L(g,f)- \chi_{\mathcal{E}}(g); \ g \in \mathcal{H},   f \in \mathcal{H} \Big\}\\
 &=(L+\chi_{\mathcal{E}})^*(f_0, \Pi v)\\
&=\inf \Big\{ L^*(f_0-r, \Pi v)+\chi_{\mathcal{E}}^*(r); \, r \in \mathcal{H} \Big\}
\end{align*}
where we have used that the Legendre dual of the sum is inf-convolution. Finally, taking into account the expression for $\chi_{\mathcal{E}}^*$ we obtain 
\begin{align*}
\mathcal{L}^*(q,v)&=\inf \Big\{ L^*( f_0-r, \Pi v); \, r \in \mathcal{H}, \Pi ^*r=0 \Big\}\\&=\inf \Big\{ L(\Pi v, f_0-r ); \, r \in \mathcal{H}, \Pi^*r=0 \Big\}\\
&=\inf \Big\{ L(\Pi v,f ); \, f \in \mathcal{H}, \Pi^*f=q \Big\}\\
&=\mathcal{L}(v,q).
\end{align*}
\end{proof}
\noindent{\bf Proof of Lemma \ref{sd-div-lag}:}
This is now a direct application of Lemma \ref{abstract lemma}. First, lift the random Lagrangian to define a self-dual Lagrangian on $L^2(\Omega_T; L^2(D;\R^n)) \times  L^2(\Omega_T; L^2(D;\R^n))$, via 
\[
\mathcal{L}(u,p)=\E \int_0^T \int_D L (u(t,x), p(t,x))\, dx \, dt,
\]
then use Lemma \ref{abstract lemma} with this Lagrangian and the operators
\begin{equation*}\label{Delta-iso}
L^2(\Omega_T; H^1_0(D)) \xrightarrow{\Pi=\nabla}  L^2(\Omega_T; L^2(D;\R^n) ) \xrightarrow{\Pi^*=\nabla^*}  L^2(\Omega_T; H^{-1}(D)), 
\end{equation*}
to get that $\mathscr{L}$ is a self-dual Lagrangian on $L^2(\Omega_T; H^1_0(D)) \times  L^2(\Omega_T; H^{-1}(D))$. Note that $\Pi^*\Pi=\nabla^*\nabla =-\Delta$ induces an isomorphism from 
$L^2(\Omega_T; H^1_0(D))$ to $L^2(\Omega_T; H^{-1}(D))$. \\

\noindent{\bf Proof of Theorem \ref{last}:}
Again, by Theorem \ref{Fitz} and the discussion in Section \ref{mon.rep}, one can associate 
to the maximal monotone map $\beta_{\omega,t}$, an $\Omega_T$-dependent self-dual Lagrangian $L_{\beta_{\omega,t}}(u,p)$ on $\R^n \times \R^n$ in such a way that 
\begin{equation*}
\beta_{\omega,t}=\bar{\partial}L_{\beta_{\omega,t}}.
\end{equation*}
If $\beta$ satisfies (\ref{A-cond20}), then the $\Omega_T$-dependent self-dual Lagrangian $L_{\beta_{\omega,t}}$ on $\R^n\times \R^n$ satisfy  for almost every $t\in [0, T]$, $\P$-a.s.
\begin{align}\label{L-inside-bounds}
C_1 (\Vr x \Vr^2_{\R^n}+\Vr p \Vr^2_{\R^n}-n_1) \leq L_{\beta_{w,t}} (x,p)
\leq C_2(\Vr x \Vr^2_{\R^n}+\Vr p \Vr^2_{\R^n} +n_2),
\end{align}
where $C_1, C_2 \in L^\infty(\Omega_T)$ and $n_1,n_2\in L^1(\Omega_T)$. \\
We can then lift it to the space 
$L^2(\Omega_T; L^2_{\R^n}(D)) \times  L^2(\Omega_T; L^2_{\R^n}(D))$ via 
\begin{equation*}\label{lifted-Lag-maxmon}
\mathcal{L}_\beta (u,p)=\E \int_0^T \int_DL_{\beta_{\omega,t}}(u(t, w, x),p(t, w, x))\, dx dt,
\end{equation*}
in such a way that for positive constants $C_1, C_2$ and $C_3$ (different from above)
\begin{align*}\label{L-bounds}
 C_2 (\Vr u \Vr^2_{L^2_H(\Omega_T)}+\Vr p \Vr^2_{L^2_H(\Omega_T)}-1) \leq  \mathcal{L}_{\beta}(u,p)
\leq C_1(1+\Vr u \Vr^2_{L^2_H(\Omega_T)}+\Vr p \Vr^2_{L^2_H(\Omega_T)}), 
\end{align*}
where $H:=L^2_{\R^n}(D)$.
In view of (\ref{A-cond10}), we also have
\begin{equation*}\label{del-bound}
\Vr \bar{\partial}{\cal L}_\beta (u)\Vr_{L^2_H(\Omega_T)} \leq C_3(1+ \Vr u \Vr_{L^2_H(\Omega_T)}).
\end{equation*}
Use now Lemma \ref{sd-div-lag} to lift  $\mathcal{L}_\beta$ to a self-dual Lagrangian $\mathscr{L}_\beta$ on $L^2(\Omega_T; H^1_0(D)) \times  L^2(\Omega_T; H^{-1}(D))$, via the formula 
\begin{eqnarray}\label{div-lag.1}
\nonumber
\mathscr{L}_\beta (u,p)&=& \inf \left\{ \E \int_0^T \int_D L_{\beta_{w,t}}(\nabla u(t,x), f(t, x))\, dx \, dt; \ f \in  L^2(\Omega_T; L_{\R^n}^2(D)), -\ntx{div} (f)=p  \right\}\\
&=&\inf \left\{\mathcal{L}_\beta (\nabla u, f); \ f \in  L^2(\Omega_T; L_{\R^n}^2(D)), -\ntx{div} (f)=p  \right\}.
\end{eqnarray}
Apply now Theorem \ref{last-thm} to get a process $v \in \mathcal{Y}^2_{H^1_0(D)}$ such that 
\begin{align*}
\mathscr{L}_\beta (v, -\tilde{v})+\langle v, \tilde{v} \rangle&=0\\
F_v&=B\\
v(0)&=u_0, 
\end{align*}
and note that 
\begin{align*}
0&= \mathscr{L}_\beta (v, -\tilde{v})+\langle v, \tilde{v} \rangle \\
&=\underset{ f \in  L^2(\Omega_T;L^2_{\R^n}(D))}\inf \bigg\{ \E \int_0^T \int_D L_{\beta (w,t)}(\nabla v,f)\, dx \, dt; \ntx{div} (f)=\tilde{v}  \bigg\}+\E \int_0^T \langle v(t), \tilde{v}(t) \rangle_{_{H^1_0,H^{-1}}} dt \\
& =\underset{ f \in  L^2(\Omega_T;L^2_{\R^n}(D))}\inf \bigg\{ \E \int_0^T \int_D L_{\beta (w,t)}(\nabla v,f)  - \langle \nabla v(x,t), {f}(x,t) \rangle \, dx \, dt\bigg\}\\
&=\underset{ f \in  L^2(\Omega_T;L^2_{\R^n}(D))}\inf J_v(f),
\end{align*}
where 
$$J_v(f):=\E \int_0^T \int_D \{L_{\beta (w,t)}(\nabla v,f)  - \langle \nabla v(x,t), {f}(x,t) \rangle\} \, dx \, dt.$$
Note that condition (\ref{L-inside-bounds}) implies that $L(y,0) \leq C(1+\Vr y\Vr_{\R^n}^2)$, which means that $J_v$ 
 is coercive on $L^2(\Omega_T;L^2_{\R^n}(D))$, thus there exists $\bar{f} \in L^2(\Omega_T;L^2_{\R^n}(D))$ with $\div(\bar{f})=\tilde{v}$ such that 
$$ \E \int_0^T \int_D L_{\beta (w,t)}(\nabla v,\bar{f})  - \langle \nabla v(x,t), \bar{f}(x,t) \rangle \, dx \, dt=0.$$
The self-duality of $L$ then implies that 
$\bar{f}(x,t) = \bar{\partial}L_\beta(\nabla v(x,t))=\beta(\nabla v(x,t))$. Taking divergence 
leads to $\tilde{v} \in \div\left(\beta(\nabla v)\right)$. Taking integrals over $[0,t]$ and using the fact that $v \in \mathcal{Y}_{H^1_0(D)}^2$ finally gives
\begin{align*}
\int_0^t \div\left(\beta(\nabla v(s))\right)\, ds&= \int_0^t \tilde{v}(s) ds = v(t)-v(0)-\int_0^t F_v(s) dW(s)\\
&= v(t)-u_0-\int_0^t B(v(s)) dW,
\end{align*}
which completes the proof.


\end{document}